\documentclass[review,hidelinks,onefignum,onetabnum]{siamart220329}

\usepackage{lipsum}
\usepackage{amsfonts}
\usepackage{graphicx}
\usepackage{epstopdf}
\usepackage{algorithmic}
\usepackage{tikz}
\usepackage{diagbox}
\usepackage{extarrows,textcomp}
\usepackage{graphicx,subfigure}
\usepackage{diagbox,multirow}
\usepackage{algorithm,algorithmic}
\usepackage{bm}
\usepackage{xcolor}
\usepackage{setspace}

\ifpdf
  \DeclareGraphicsExtensions{.eps,.pdf,.png,.jpg}
\else
  \DeclareGraphicsExtensions{.eps}
\fi

\newsiamremark{remark}{Remark}
\newsiamremark{hypothesis}{Hypothesis}
\crefname{hypothesis}{Hypothesis}{Hypotheses}
\newsiamthm{claim}{Claim}

\numberwithin{equation}{section}
\numberwithin{equation}{section}

\DeclareMathOperator*{\argmin}{arg\,min}

\newcommand{\dd}{\,{\rm d}}

\newcommand{\bs}{\boldsymbol}

\def\bu{{\boldsymbol u}}
\def\bn{{\boldsymbol n}}
\def\be{{\boldsymbol e}}
\def\bv{{\boldsymbol v}}
\def\bw{{\boldsymbol w}}
\def\bx{{\boldsymbol x}}
\def\dk{{\boldsymbol K}}
\def\dd{{\boldsymbol D}}
\def\T{{\mathcal{T}}}
\def\N{{\mathcal{N}}}
\def\bT{{\boldsymbol{T}}}
\def\bxi{{\boldsymbol{\xi}}}
\def\bta{{\boldsymbol{\eta}}}
\def\btheta{{\boldsymbol{\theta}}}

\title{Newton's method and its hybrid with machine learning for Navier-Stokes Darcy Models discretized by mixed element methods\thanks{Submitted to the editors DATE.
\funding{The  work  of  the  first  author  was  supported  by the National Natural Science Foundation of China (Grant No. 12071289), and the Strategic Priority Research Program of the Chinese Academy of Sciences (Grant No. XDA25010402). The work of the second author was supported by the National Natural Science Foundation of China (Grant No. 12301519)}}}

\author{Jianguo~Huang\thanks{School of Mathematical Sciences, Shanghai Jiao Tong University, Shanghai 200240, China
  (\email{jghuang@sjtu.edu.cn}, \email{penghui23@sjtu.edu.cn},
   \email{wu1150132305@sjtu.edu.cn}).}
\and Hui Peng\footnotemark[2]
\and Haohao Wu\footnotemark[2]}

\ifpdf

\begin{document}

\maketitle

\begin{abstract}
This paper focuses on discussing Newton's method and its hybrid with machine learning for the steady state Navier-Stokes Darcy model discretized by mixed element methods. First, a Newton iterative method is introduced for solving the relative discretized problem. It is proved technically that this method converges quadratically with the convergence rate independent of the finite element mesh size, under certain standard conditions. Later on, a deep learning algorithm is proposed for solving this nonlinear coupled problem. Following the ideas of an earlier work by Huang, Wang and Yang (2020), an Int-Deep algorithm is constructed by combining the previous two methods so as to further improve the computational efficiency and robustness. A series of numerical examples are reported to show the numerical performance of the proposed methods.
\end{abstract}



\section{Introduction}
The Navier-Stokes Darcy model is frequently encountered in various industrial engineering scenarios, for instance, in groundwater \cite{Cao-2010-finite, Discacciati-2002-math, Layton-2013-ana},  flow in porous media \cite{Arbogast-2006-Hom, Arbogast-2007-com}, industrial filtrations \cite{Hanspal-2006-num} and so on. Due to its importance, many numerical methods have been developed and analyzed for the coupled nonlinear model, which can be roughly divided into two categories. The first one focuses on numerically solving the coupled system directly. Historically, Girault and Rivi\'{e}re \cite{DG2009} propose and analyze a discontinuous Galerkin method for the Navier-Stokes Darcy model. Following \cite{DG2009}, Chidyagwai and Rivi\'{e}re \cite{Chidyagwai-2009-on} develop a hybrid coupled method, where the continuous finite elements are employed in the free flow region and the discontinuous Galerkin finite elements in the porous medium region, respectively. Wu and Mei \cite{Wu-2016-non} discretize the model with a non-conforming finite volume element method. Cesmelioglu and Rhebergen present and analyze in \cite{Cesmelioglu2023hybrid} a strongly conservative hybridizable discontinuous Galerkin scheme. The other strategy is to decouple the model first and then numerically solve two subproblems separately. 
In \cite{xujinchao2009}, Cai, Mu and Xu propose a decoupled and linearized two-grid algorithm. The two-grid decoupled method is further studied in \cite{Zuo2015, jia-2016-mod}. In addition, Cai, Huang and Mu also develop a multi-grid algorithm in \cite{2018-cai-some}. He et al. \cite{He-2015-domain} design a domain decomposition method for the Navier-Stokes Darcy model with the BJ interface condition. In order to further improve computational efficiency, {Du et al. develop in \cite{Wang-2021-novel, Du-2022-local} a series of parallel algorithms based on decoupled model. We mention that all the previous methods are used to solve steady-state Navier-Stokes Darcy models. In fact, there are deep and through works on numerical solution for unsteady state models, and we skip the details due to the limit of space.


Both the direct and decoupling methods require solving nonlinear system of equations globally or locally through iterative methods, e.g. Newton's method. However, there are few works on investigating the global convergence analysis of these iterative methods. As far as we know, only Badea et al. \cite{Badea-2010-num} study Newton's method for the Navier-Stokes Darcy model in the case of infinite dimensions. The analysis is rather involved; the key techniques are first reformulating the coupled nonlinear problem as an interface equation and then developing the underlying convergence analysis of Newton iterative method by means of the Kantorovich theorem. However, when applying the previous arguments to the nonlinear system arising from discretization of the Navier-Stokes Darcy model by mixed element methods, it is very hard to show the convergence rate is uniformly bounded with respect to the finite element mesh size $h$. This study is important for real applications. In fact, if the convergence rate goes to $1$ when $h$ goes to zero, then we even cannot observe convergence of the underlying iterative method due to the rounding-off error. The other point to be emphasized is that Newton's method is locally convergent, so it is very challenging to develop a strategy on the choice of initial guesses to ensure convergence. For our problem under discussion, a common way is to set the initial guess as the numerical solution of the corresponding linear Stokes Darcy problem. However, such a choice is not robust with respect to the model coefficients; we refer to numerical results in subsection \ref{alg:int-deep} for details. Therefore, it's necessary to develop an effective approach to choosing initial guesses in order to improve the computational performance and robustness of Newton iterative method.

With the above discussion in mind, our study in this paper is twofold. First of all,  we analyze in detail the convergence of Newton iterative method for the Navier-Stokes Darcy model discretized by mixed finite element methods. Different from the approach in \cite{Badea-2010-num}, we derive the required result (cf. Theorem \ref{convergence-rate}) in another way. Our analysis relies on a key observation which indicates that a critical interface term in the Navier-Stokes equations, which is connected with Darcy equations, can be decomposed into a positive term plus an easily estimated term. Combining this finding with the analysis of Newton's iteration for Navier-Stokes model in \cite{He-2009-con}, we prove that the iterative method (cf. \eqref{New-alg-1}-\eqref{New-alg-3}) is convergent quadratically under certain standard conditions, with the convergence rate independent of the finite element mesh size $h$.

On the other hand, it is observed that a deep learning initialized iterative (Int-Deep) method is introduced in \cite{intdeep2020} for nonlinear variational problems. It is a hybrid iterative method, consisting of two phases. In the first phase, an expectation minimization problem formulated for a given nonlinear partial differential equations (PDE) is solved by deep learning (DL) methods. In the second phase, the numerical solution from the first phase is interpolated as the initial guess, and some Newton-type iterative methods are used to solve the finite-dimensional problem discretized by finite element methods, which converges rapidly to the finite element solution. This method has been used successfully to solve semi-linear elliptic problems, linear and nonlinear eigenvalue problems (cf. \cite{intdeep2020}). The success of the method may rely partly on the so-called frequency principle mentioned in \cite{Xu-2020-frequency, Xu-2019-frequency}, which asserts that a deep neural network (DNN) tends to fit training data by a low-frequency function. This result implies that the DL solution with few iteration steps may capture the low frequency components of the exact solution, so it is reasonable to act as an initial guess of an iterative method for a variational problem.

Thus, the second goal of this paper is to use the ideas of the above Int-Deep method to solve the Navier-Stokes Darcy model combined with Newton iterative method just mentioned. To this end, we first design a physics informed neural networks (PINN)-type \cite{Karniadakis2019,Karniadakis2021} DL algorithm for the Navier-Stokes Darcy model, through expressing the unknowns with ResNet functions \cite{2016-Resnet} and constructing a log-loss function (see \eqref{eq:3.1}) to improve computational efficiency. Furthermore, following the ideas in \cite{intdeep2020}, we introduce the Int-Deep
algorithm for the discrete Navier-Stokes Darcy model (cf. \eqref{FEM-1}-\eqref{FEM-2}). In other words, we solve the discrete problem by means of Newton iterative method, with the interpolant of the DNN solution as an initial guess. Finally, we perform a series of numerical examples to confirm that this Int-Deep algorithm is able to reach the accuracy of the finite element method with few iteration steps. In particular, the algorithm is shown to be robust with respect to some physical parameters, e.g., the viscosity coefficient of fluid and the hydraulic conductivity tensor of the porous medium.

The rest of the paper is organized as follows. In Section \ref{sec:model}, we introduce the model problem and recall some results of the finite element method for the nonlinear coupled system. In Section \ref{sec:newton}, we discuss the stability and convergence of  Newton iterative method under some conditions. In Section \ref{sec:algo}, a DL algorithm for the Navier-Stokes Darcy model is designed, and then the Int-Deep algorithm is developed for the Navier-Stokes Darcy model. In Section \ref{sec:test}, we present some numerical examples to demonstrate the accuracy and efficiency of the Int-Deep algorithm.

\section{The Navier-Stokes Darcy problem and its discretization}\label{sec:model}
In this section, we aim to introduce the Navier-Stokes Darcy model and its finite element discretization. We also recall some important existing results which will be frequently used  later on.
\subsection{Navier-Stokes Darcy model}
Let $\Omega \subset R^{d}~(d=2,~3)$ be a bounded polytopal domain, which is subdivided into a free fluid region $\Omega_{f}$ and a porous media region $\Omega_{p}$ by an interface $\Gamma$, as shown in Figure \ref{fig:1}.
Let $\Gamma_f = \partial\Omega_f\setminus \Gamma$ and $\Gamma_p = \partial\Omega_p\setminus \Gamma$ be the outer boundaries of the domains $\Omega_f$ and $\Omega_p$, respectively. Denote by
$\boldsymbol{n}_{f}$ and $\boldsymbol{n}_{p}$ the unit outward normal vectors on $\partial \Omega_{f}$ and $\partial \Omega_{p}$, respectively. The unit tangential vector on $\Gamma$ is denoted by $\boldsymbol{\tau}$.
\begin{figure}[htb]\begin{center}\setlength\unitlength{1in}
    \begin{picture}(3,2.5)
 \put(0,0){\includegraphics[height=2.5in,width=2.2in]{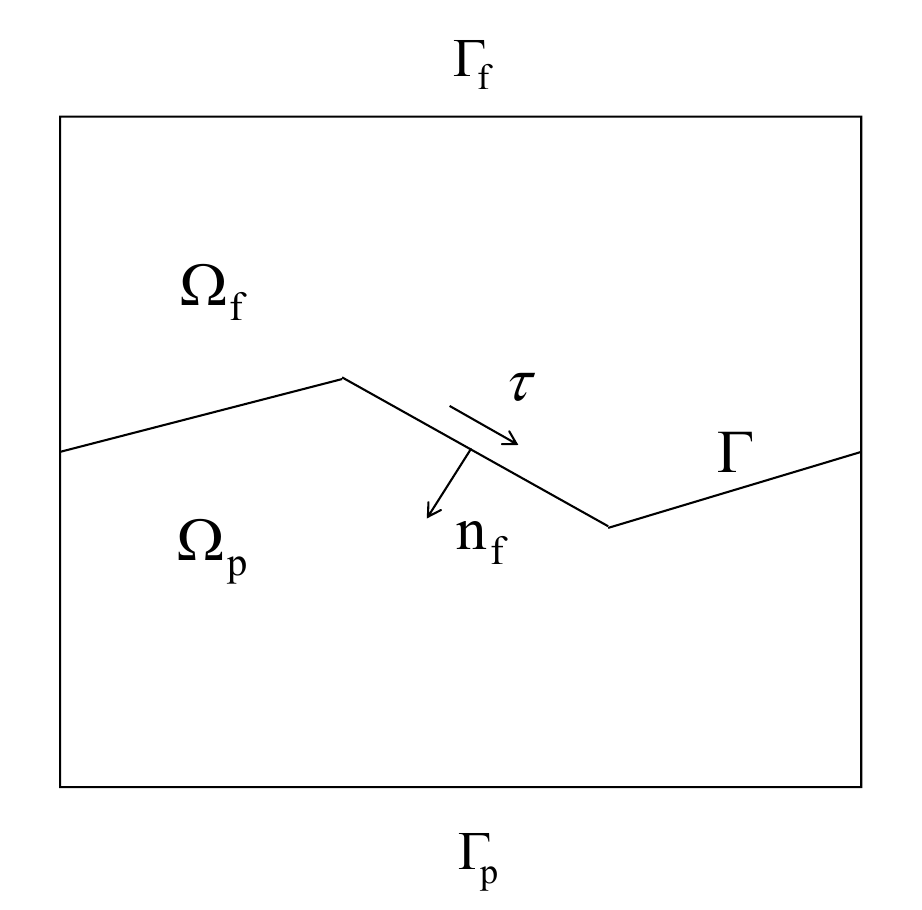}}
    \end{picture}
\caption{ {\color{black}Domain schematic for Naiver-Stokes Darcy coupled flow.}  } \label{fig:1}
\end{center}
\end{figure}

The free flow in $\Omega_{f}$ is governed by the steady Navier-Stokes equations:
\begin{equation}\label{model-pro1}
  \begin{cases}
  -\operatorname{div}(\boldsymbol{T}(\boldsymbol{u}, p))+\rho(\boldsymbol{u} \cdot \nabla) \boldsymbol{u}=\boldsymbol{f}_f & \text{in}~ \Omega_{f}, \\ \operatorname{div} \boldsymbol{u}=0  &  \text{in}~ \Omega_{f},
  \end{cases}
\end{equation}
where $\boldsymbol{u}$ and $p$ indicate the velocity field and the pressure field of the flow, respectively;
$\nu>0$ stands for the viscosity coefficient,
$
\boldsymbol{T}(\boldsymbol{u}, p)=2 \nu \boldsymbol{D}(\boldsymbol{u})-p \boldsymbol{I}
$
is the stress tensor, and
$
\boldsymbol{D}(\boldsymbol{u})=\frac{1}{2}\left(\nabla \boldsymbol{u}+\nabla^{{\rm T}} \boldsymbol{u}\right)
$
is the deformation rate tensor; $\rho$ denotes the density of the flow and $\boldsymbol{f}_{f}$ is the external force.

The flow motion in the porous media region $\Omega_{p}$ is described by Darcy's law:
\begin{equation}{\label{model-pro2}}
-\operatorname{div}(\bs{K}\nabla \varphi)=f_p \quad  \text{in}~\Omega_p,
\end{equation}
where
$\varphi$ is the hydraulic head, $f_p$ is a source term and $\bs{K}$ represents the hydraulic conductivity tensor. In this paper, we assume that $\bs{K}$ is symmetric $K_{i j}=K_{j i}$, and positive definite, i.e.
$$
\alpha_{1}(\boldsymbol{x}, \boldsymbol{x}) \leq(\boldsymbol{K} \boldsymbol{x}, \boldsymbol{x}) \leq \alpha_{2}(\boldsymbol{x}, \boldsymbol{x}) \quad \forall\; \boldsymbol{x} \in ~\Omega_{p},
$$
for two positive constants $\alpha_{1}$ and $\alpha_{2}$.
\vskip 0.2cm

The physical quantities are coupled on the interface $\Gamma$ through three interface conditions:
    \begin{equation}\label{model-pro3}
     \begin{cases}
    \boldsymbol{u} \cdot \boldsymbol{n}_{f}=-\boldsymbol{K} \nabla \varphi \cdot \boldsymbol{n}_{f} ,\\
    {\left[-\boldsymbol{T}(\boldsymbol{u}, p) \boldsymbol{n}_{f}\right] \cdot \boldsymbol{n}_{f}=\rho g \varphi},\\
    {\left[-\boldsymbol{T}(\boldsymbol{u}, p)  \boldsymbol{n}_{f}\right] \cdot \boldsymbol{\tau}_{j}=\frac{\nu \alpha}{\sqrt{\boldsymbol{\tau}_{j} \cdot \nu \boldsymbol{K} \cdot \boldsymbol{\tau}_{j}}} \boldsymbol{u} \cdot \boldsymbol{\tau}_{j},\; j=1, \ldots, d-1},
    \end{cases}
    \end{equation}
where $\left\{\boldsymbol{\tau}_{j}\right\}_{j=1}^{d-1}$ are linearly independent unit tangential vectors on $\Gamma$, and $\alpha$ is an experimentally determined positive parameter.

For simplicity, we consider homogeneous Dirichlet boundary conditions on the outer boundaries:
\begin{equation}\label{model-pro4}
 \begin{cases}
\boldsymbol{u}=\boldsymbol{0} & \text { on } \Gamma_{f}, \\ \varphi=0 & \text { on } \Gamma_{p} .
\end{cases}
\end{equation}

In order to describe the variational formulation of this model, we have to introduce some function spaces in advance.
Denote the admissible spaces for velocity, pressure and hydraulic head, respectively by
$$
\begin{aligned}
\boldsymbol{X}_{f} &=\Big\{\boldsymbol{v} \in \boldsymbol{H}^{1}(\Omega_{f})=[H^{1}(\Omega_{f})]^{d} \mid \boldsymbol{v}=\boldsymbol{0} \text { on } \Gamma_{f}\Big\},\qquad   Q=L^{2}\left(\Omega_{f}\right), \\
X_{p} &=\left\{\psi \in H^{1}\left(\Omega_{p}\right) \mid \psi=0 \text { on } \Gamma_{p}\right\}.
\end{aligned}
$$
 Then, the variational formulation of the coupled Navier-Stokes Darcy model $(\ref{model-pro1})-(\ref{model-pro4})$ is given as follows.

Find $\bs{u} \in \bs{X}_{f},\varphi\in X_{p}$ and $q\in Q $ such that
\begin{equation}\label{weak-form-1}
\begin{cases}
 &a_f(\boldsymbol{u},\boldsymbol{v})+a_p(\varphi,\psi)+a_{\Gamma}(\boldsymbol{u},\psi;\boldsymbol{v},\varphi)+c(\boldsymbol{u};\boldsymbol{u},\boldsymbol{v})+b(\boldsymbol{v},p)\\
 &=(\boldsymbol{f}_f,\boldsymbol{v})+\rho g(f_p,\psi)~~~~\forall~ \boldsymbol{v}\in \bs{X}_{f}, \psi \in X_{p},\\
&b(\boldsymbol{u},q)=0~~~~\forall ~ q \in Q,
\end{cases}
\end{equation}
where,
$$
\begin{aligned}
 & a_{f}(\boldsymbol{u}, \boldsymbol{v})= \int_{\Omega_{f}} 2\nu \boldsymbol{D}(\boldsymbol{u}):\boldsymbol{D}(\boldsymbol{v})+\sum_{j=1}^{d-1} \frac{\nu \alpha}{\sqrt{ \boldsymbol{\tau}_{j} \cdot\nu \boldsymbol{K} \boldsymbol{\tau}_{j}}} \int_{\Gamma}\left(\boldsymbol{u} \cdot \boldsymbol{\tau}_{j}\right)\left(\boldsymbol{v} \cdot \boldsymbol{\tau}_{j}\right), \\
&a_{p}(\varphi, \psi)=\rho g \int_{\Omega_{p}}  \boldsymbol{K} \nabla \varphi \cdot\nabla \psi,\quad
a_{\Gamma}(\bu,\psi;\bv,\varphi)=\rho g \int_{\Gamma}(\varphi \boldsymbol{v}-\psi \boldsymbol{u}) \cdot \boldsymbol{n}_{f},
\\
&c(\boldsymbol{u}; \boldsymbol{v}, \boldsymbol{w})=\rho \int_{\Omega_{f}}(\boldsymbol{u} \cdot \nabla) \boldsymbol{v} \cdot \boldsymbol{w},
\quad b(\boldsymbol{v}, p)=-\int_{\Omega_{f}} p \operatorname{div} \boldsymbol{v},\\
&(\boldsymbol{f}_{f},\boldsymbol{v})=\int_{\Omega_{f}} \boldsymbol{f}_{f} \cdot \boldsymbol{v}, \quad (f_p,\psi)= \int_{\Omega_{p}} f_{p} \psi.
\end{aligned}
$$
Here and hereafter, we omit the infinitesimal element over an integration symbol when there is no confusion caused.

The existence and uniqueness of solution of $(\ref{weak-form-1})$  can be found in \cite{DG2009}.

\subsection{Finite element approximation}

 In this subsection, we shall give the finite element discretization of the Navier-Stokes Darcy model and review some existing results for subsequent analysis.

Let $\T_{f,h}$ (resp. $\T_{p,h}$) be a nondegenerate quasi-uniform triangulation of $\Omega_f$ (resp. $\Omega_p$). The partition $\T_{f,h}$ matches with the partition $\T_{p,h}$ at the interface $\Gamma$. Let $\T_h=\T_{f,h}\cup \T_{p,h}$ be the partition of the whole domain $\Omega$. For each element $T\in\T_h$, we denote by $h_T$ the diameter of $T$; the mesh size of $\T_h$ is defined by $h=\max_{T\in\T_h}h_T$. Let $k$ be a non-negative integer, and denote by $\mathbb{P}_k(T)$ the set of polynomials with the degree no more than $k$ on element $T\in \T_h$. Moreover, we use $C$ (with or without subscripts) to denote a generic positive constant independent of the mesh size $h$, which may take different values at different occurrences.

There are various combinations of the stable Stokes finite element pairs \cite{Boffi-2013-mixed, hughes-1987-finite} (such as MINI element, Taylor-Hood element, conforming Crouzeix-Raviart element) 
with continuous finite elements for Darcy equations. For convenience, we focus on the traditional Taylor-Hood element as an typical example in this paper. It is not difficult to extend the corresponding theory to other elements.

Define some finite element spaces by
\begin{eqnarray*}
\boldsymbol{X}_{f,h}&=&\{\boldsymbol{v}_h\in \boldsymbol{X}_f,~\boldsymbol{v}_h|_{T}\in [\mathbb{P}_{2}(T)]^d~\forall~T\in\T_{f,h}\},\\
Q_h&=&\{q_h\in Q\cap C^0(\Omega),~q_h|_T\in \mathbb{P}_1(T)~\forall~T\in\T_{f,h}\}
,\\
X_{p,h}&=&\{\psi_h\in X_p,~\psi_h|_T\in \mathbb{P}_2(T)~\forall~T\in\T_{p,h}\}.
\end{eqnarray*}

With these discrete function spaces, we give the finite element method for the coupled Navier-Stokes Darcy
model $(\ref{model-pro1})-(\ref{model-pro4})$: Find $\boldsymbol{u}_h\in \boldsymbol{X}_{f,h}$, $p_h\in Q_h$, and $\varphi_h\in X_{p,h}$ such that
\begin{align}\label{FEM-1}
&a_f(\boldsymbol{u}_h,\boldsymbol{v}_h)+a_p(\varphi_h,\psi_h)+a_{\Gamma}(\boldsymbol{u}_h,\psi_h;\boldsymbol{v}_h,\varphi_h)+c(\boldsymbol{u}_h;\boldsymbol{u}_h,\boldsymbol{v}_h)+b(\boldsymbol{v}_h,p_h)\\\nonumber
&=(\boldsymbol{f}_f,\boldsymbol{v}_h)+\rho g(f_p,\psi_h)~~~~\forall~ \boldsymbol{v}_h\in \bs{X}_{f,h},\;\psi_h \in X_{p,h},\\\label{FEM-2}
&b(\boldsymbol{u}_h,q_h)=0~~~~\forall ~ q_h \in Q_h.
\end{align}

To proceed with the forthcoming discussion about the well-posedness and convergence of the finite element method, 
we first recall some basic inequalities in Sobolev spaces as follows \cite{Adams-1975-sob}. There exist constants $C_p,~C_t, ~C_k$ and $C_s$ only depending on the domain $\Omega_f$, and $\tilde{C}_p$ and $\tilde{C}_t$ only depending on the domain $\Omega_p$, such that for all $\bv\in \mathbf{X}_f$ and $\psi\in X_p$,
\begin{eqnarray*}
(\text{the Poincar\'e inequality})~~~~~~&&\|\bv\|_{0,\Omega_f}\leq C_p|\bv|_{1,\Omega_f},\hspace{1.5cm} \|\psi\|_{0,\Omega_p}\leq \tilde{C}_p|\psi|_{1,\Omega_p}.\\
(\text{the trace inequality})\hspace{1.3cm}&&\|\bv\|_{0,\Gamma}\leq C_t|\bv|_{1,\Omega_f},\hspace{1.75cm} \|\psi\|_{0,\Gamma}\leq \tilde{C}_t|\psi|_{1,\Omega_p}.\\
(\text{Korn's inequality})\hspace{1.7cm}&&|\bv|_{1,\Omega_f}\leq C_k\|\boldsymbol{D}(\bv)\|_{0,\Omega_f}.\\
(\text{the Sobolev inequality})\hspace{0.95cm}&&\|\bv\|_{L^6(\Omega_f)}\leq C_s|\nabla\bv|_{1,\Omega_f}.
\end{eqnarray*}

Using the H{\"{o}lder inequality, the Sobolev inequality and Korn's inequality, we immediately have (cf. \cite{Temam-1979-Navier})
\begin{align}\label{tri-est}
c(\bu;\bv,\bw)\leq \mathcal{N}_d\|\dd(\bu)\|_{0,\Omega_f}\|\dd(\bv)\|_{0,\Omega_f}\|\dd(\bw)\|_{0,\Omega_f},
\end{align}
with $\N_d$ depending only on $C_s$, $C_k$ and $d$.

In order to describe some important results, we define
\begin{align*}
\|\bm{f}_f\|_*=\frac{(\bm{f}_f,\bv)}{\|\dd(\bv)\|_{0,\Omega_f}},\quad \|f_p\|_{*}=\frac{(f_p,\psi)}{|\psi|_{1,\Omega_p}}.
\end{align*}

The theoretical results given below indicate the well-posedness and convergence of the finite element method $(\ref{FEM-1})-(\ref{FEM-2})$. We refer the interesting readers to $\cite{xujinchao2009}$ for details along this line.  

\begin{theorem}\label{FEM-well-pose}
Assume that $\bm{f}_f\in [L^2(\Omega_f)]^d$, $f_p\in L^2(\Omega_p)$ and the viscosity coefficient $\nu$ satisfies
\begin{align}\label{assume-1}
\nu^{\frac32}\geq \sqrt{2}\mathcal{N}_d\mathcal{R},
\end{align}
with
\begin{align}
\label{assume-1a}
\mathcal{R}=\left(\frac{1}{\nu}\|\bm{f}_f\|^2_*+\frac{\rho g}{\alpha_1}\|f_p\|^2_*\right)^{\frac12}.
\end{align}
Then the finite element scheme $(\ref{FEM-1})-(\ref{FEM-2})$ has a unique solution satisfying the following estimate
\begin{align}\label{FEM-priori-est1}
2\nu\|\boldsymbol{D}(\bu_{h})\|^2_{0,\Omega_f}+\rho g\|\boldsymbol{K}\nabla\varphi_{h}\|^2_{0,\Omega_p}\leq \mathcal{R}^2.
\end{align}
\end{theorem}
\begin{remark}
Under the assumption \eqref{assume-1}, it follows from \eqref{assume-1a} that
\begin{align}\label{FEM-priori-est2}
\|\dd(\bu_h)\|_{0,\Omega_f}\leq \frac{\nu}{2\N_d}.
\end{align}
\end{remark}

\begin{lemma}
There is a constant $\beta>0$ independent of $h$ such that
\begin{align}\label{inf-sup}
\inf_{q_h\in Q_h}\sup_{\boldsymbol{v}_h\in \boldsymbol{X}_{f,h}}\frac{b(\boldsymbol{v}_{s,h},q_h)}{\|\dd(\boldsymbol{v}_h)\|_{0,\Omega_f}\|q_h\|_{0,\Omega_f}}\geq \beta.
\end{align}
\end{lemma}

\begin{theorem}{\label{FEM-err-est}}
Let $(\bu, p, \varphi)$ be the solution of problem $(\ref{model-pro1})-(\ref{model-pro4})$ and $(\bu_h, p_h,$ $\varphi_h)$ be the solution of finite element scheme $(\ref{FEM-1})-(\ref{FEM-2})$. Assume that $\bu\in \boldsymbol{H}^{3}(\Omega_{f})$, $p\in H^2(\Omega_f)$ and $\varphi\in H^3(\Omega_p)$. Then the following estimates hold:
\begin{align}\nonumber
2\nu&\|\dd(\bu-\bu_h)\|^2_{0,\Omega_f}+\sum_{j=1}^{d-1}\|(\bu-\bu_h)\cdot\bm{\tau}_j\|^2_{0,\Gamma}+\rho g \|\dk^{\frac12}\nabla(\varphi-\varphi_h)\|^2_{0,\Omega_p}\\
&\leq  Ch^4(\|\bu\|^2_{3,\Omega_f}+\|p\|^2_{2,\Omega_f}+\|\varphi\|^2_{3,\Omega_p}),\\
\|p&-p_h\|^2\leq Ch^4(\|\bu\|^2_{3,\Omega_f}+\|p\|^2_{2,\Omega_f}+\|\varphi\|^2_{3,\Omega_p}).
\end{align}
\end{theorem}

\section{Newton's Method for the Navier-Stokes Darcy model}\label{sec:newton}

In this section, we give Newton iterative method for solving the Navier-Stokes Darcy model under the finite element discretization 
and provide stability and convergence analysis of the iterative method.

According to Newton iterative method for Navier-Stokes Darcy problem in infinite dimension (cf. \cite{Badea-2010-num}), its finite element analogue can be described as follows.  

Given $\bu^h_0\in X_{f,h}$, for $n\geq 1$, find $\bu^h_n\in \boldsymbol{X}_{f,h}$, $p^h_n\in Q_h$ and $\varphi^h_n\in X_{p,h}$ such that
\begin{align}\nonumber
&a_f(\bu^h_n,\bv_h)+b(\bv_h,p^h_n)+\rho g \langle\varphi^h_n, \bv_h\cdot\bn_f\rangle_{\Gamma}+c(\bu^h_n;\bu^h_{n-1},\bv_h)
+c(\bu^h_{n-1};\bu^h_n,\bv_h)\\\label{New-alg-1}
&=(\bm{f}_f,\bv_h)+c(\bu^h_{n-1};\bu^h_{n-1},\bv_h), \\\label{New-alg-2}
&b(\bu^h_n,q_h)=0, \\\label{New-alg-3}
&a_p(\varphi^h_n,\psi_h)=\rho g \langle \psi_h,\bu^h_n\cdot\bn_f\rangle_{\Gamma}+\rho g(f_p,\psi_h),
\end{align}
for any $\boldsymbol{v}_h\in \bs{X}_{f,h},\;q_h\in Q_h,\;\psi_h\in X_{p,h}$.

To analyze Newton iterative method, we next introduce two auxiliary problems given below.

For any $\zeta\in L^2(\Gamma)$ or $\eta\in L^2(\Omega_p)$, find $\phi_h\in X_{p,h}$ such that
\begin{align}\label{aux-problem-darcy}
a_p(\phi_h,\psi_h)=\rho g \langle \zeta,\psi_h\rangle_{\Gamma}\quad \forall\; \psi_h\in X_{p,h},
\end{align}
or
\begin{align}\label{aux-problem-2}
a_p(\phi_h,\psi_h)=\rho g(\eta,\psi_h)\quad \forall\; \psi_h\in X_{p,h}.
\end{align}

It is easy to prove by the Lax-Milgram lemma that problem $\eqref{aux-problem-darcy}$ or problem \eqref{aux-problem-2} has a unique solution, which induces a linear operator $\phi_h=T_h (\zeta)$ or $\phi_h=T_h (\eta)$. Furthermore, we have the following estimate which is easy to derive but important in our forthcoming analysis.
\begin{align}
\label{aux-est-1}
\|T_h(\eta)\|_{1,\Omega_p}\leq C_r\|\eta\|_*,
\end{align}
with $C_r$ a positive constant independent of $h$.

The discussion on stability and convergence will be based on the following assumptions about the initial function $\bu^h_0\in \boldsymbol{X}_{f,h}$ in Newton iterative method $\eqref{New-alg-1}-\eqref{New-alg-3}$:
\begin{align}\label{ini-con1}
2\nu\|\dd(\bu^h_0)\|_{0,\Omega_f}&\leq  C_1(\|\bm{f}_f\|_*+\|f_p\|_*),\\\label{ini-con2}
\|\dd(\bu_h-\bu^h_0)\|_{0,\Omega_f}&\leq
\frac{C_1}{\nu}\left(\|\bm{f}_f\|_*+\|f_p\|_*\right),
\end{align}	
where
$$
C_1=1+\rho g C_r\tilde{C}_{t}C_{t}C_k.
$$

\begin{lemma}\label{Newton-sta-lemma}
Assume the initial conditions \eqref{ini-con1}-\eqref{ini-con2} and  the following stability condition hold:
\begin{align}\label{Newton-assume1}
\|\bm{f}_f\|_{*}+\|f_p\|_{*}\leq \frac{\nu^2}{9C_1\N_d}.
\end{align}
Then $\bu^h_m$ defined by Newton iterative method $\eqref{New-alg-1}-\eqref{New-alg-3}$ satisfies
\begin{align}\label{Newton-sta}
\|\dd(\bu^h_m)\|_{0,\Omega_f}\leq \frac{3C_1}{2\nu}\left(\|\bm{f}_f\|_*+\|f_p\|_*\right),
\end{align}
for all $m\geq 0$.
\end{lemma}
\begin{proof}
We prove $\eqref{Newton-sta}$ by mathematical induction. The initial condition $\eqref{ini-con1}$ implies that $(\ref{Newton-sta})$ holds for $m=0$. Assuming that $(\ref{Newton-sta})$ holds for $m=J$ and we aim to prove the estimate holds for $m+1$. 
Taking $\zeta=\bu^h_{J+1}\cdot\bn_f$ in $\eqref{aux-problem-darcy}$ gives $\phi^b_{J+1}=T_h(\bu^h_{J+1}\cdot\bn_f)$. Taking $\eta=f_p$ in $\eqref{aux-problem-2}$ gives $\phi^i$, which satisfies $\|\phi^{i}\|_{1,\Omega_{p}}\leq C_r\|f_p\|_*$ from $\eqref{aux-est-1}$. According to the superposition principle, $\varphi^h_{J+1}=\phi^b_{J+1}+\phi^i$ solves $(\ref{New-alg-3})$, which implies that 
\begin{align}\nonumber
\rho g \langle \varphi^h_{J+1},\bu^h_{J+1}\cdot\bn_f\rangle_{\Gamma}&=\rho g\langle \phi^b_{J+1}, \bu^h_{J+1}\cdot\bn_f\rangle_{\Gamma}+\rho g\langle\phi^i, \bu^h_{J+1}\cdot\bn_f\rangle_{\Gamma}\\\nonumber
&= a_p(T_h(\bu^n_{J+1}\cdot\bn_f),\phi^b_{J+1})+\rho g\langle \phi^i, \bu^h_{J+1}\cdot\bn_f\rangle_{\Gamma}\\\nonumber
&= a_p(\phi^b_{J+1},\phi^b_{J+1})+\rho g\langle \phi^i, \bu^h_{J+1}\cdot\bn_f\rangle_{\Gamma}\\\label{inter-est}
&\geq \rho g\langle \phi^i, \bu^h_{J+1}\cdot\bn_f\rangle_{\Gamma}.
\end{align}
Taking $(\bv_h,q_h)=(\bu^{J+1}_h,p^{J+1}_h)\in \boldsymbol{X}_{f,h}\times Q_h$ in \eqref{New-alg-1} with $n=J+1$ and using Korn's inequality, the estimates $\eqref{tri-est}$, $\eqref{inter-est}$ and the induction assumption, we have
\begin{align*}
2\nu &\|\dd(\bu^h_{J+1})\|^2_{0,\Omega_f}\\
\leq &(\bm{f}_f,\bu^h_{J+1})-\rho g \langle \varphi^h_{J+1},\bu^h_{J+1}\cdot\bn_f\rangle_{\Gamma}\\
&+c(\bu^h_J;\bu^h_J,\bu^h_{J+1})-c(\bu^h_{J+1};\bu^h_{J},\bu^h_{J+1})-c(\bu^h_J;\bu^h_{J+1},\bu^h_{J+1})\\
\leq & \|\bm{f}_f\|_*\|\dd(\bu^h_{J+1})\|_{0,\Omega_f}-\rho g\langle\phi^i,\bu^h_{J+1}\cdot\bn_f\rangle_{\Gamma}\\
&+\N_d\|\dd(\bu^h_J)\|^2_{0,\Omega_f}\|\dd(\bu^h_{J+1})\|_{0,\Omega_f}
+2\N_d\|\dd(\bu^h_{J})\|_{0,\Omega_f}\|\dd(\bu^h_{J+1})\|^2_{0,\Omega_f}\\
\leq & C_1(\|\bm{f}_f\|_*+\|f_p\|_*)\|\dd(\bu^h_{J+1})\|_{0,\Omega_f}
+\frac{9\N_d}{4\nu^2}C_1^2(\|\bm{f}_f\|_*+\|f_p\|_*)^2\|\dd(\bu^h_{J+1})\|_{0,\Omega_f}\\
&+\frac{3\N_d}{\nu}C_1(\|\bm{f}_f\|_*+\|f_p\|_*)\|\dd(\bu^h_{J+1})\|^2_{0,\Omega_f}.
\end{align*}
This, in conjunction with $(\ref{Newton-assume1})$, gives rise to
\begin{align*}
(\nu-\frac{\nu}{6})\|\dd(\bu^h_{J+1})\|_{0,\Omega_f}\leq \frac{5}{4}C_1(\|\bm{f}_f\|_*+\|f_p\|_*),
\end{align*}
which implies $(\ref{Newton-sta})$.
\end{proof}

\begin{lemma}
The function $\bu_h$ defined by the finite element method $(\ref{FEM-1})-(\ref{FEM-2})$ satisfies
\begin{align}\label{FEM-solu-est}
\nu \|\dd(\bu_h)\|_{0,\Omega_f}\leq C_1 (\|\bm{f}_f\|_*+\|f_p\|_*).
\end{align}
\end{lemma}
\begin{proof}
Taking $\bv_h=\bu_h,~\psi_h=0$ in $(\ref{FEM-1})$ and noting that
\begin{align*}
b(\bu_h,p_h)=0,
\end{align*}
we get
\begin{align}
2\nu\|\dd(\bu_h)\|^2_{0,\Omega_f}=(\bm{f}_f,\bu_h)-c(\bu_h;\bu_h,\bu_h)-\langle \varphi_h, \bu_h\cdot\bn_f\rangle_{\Gamma}.
\end{align}
In view of the estimates $(\ref{tri-est})$ and $(\ref{FEM-priori-est2})$, the trilinear from in the above equation can be bounded as
\begin{align}\label{FEM-solu-est-1}
c(\bu_h;\bu_h,\bu_h)\leq \N_d\|\dd(\bu_h)\|^3_{0,\Omega_f}\leq \frac{\nu}{2}\|\dd(\bu_h)\|^2_{0,\Omega_f}.
\end{align}
In order to handle the interface term $\langle \varphi_h, \bu_h\cdot\bn_f\rangle_{\Gamma}$, taking $\bv_h=0$ in $(\ref{FEM-1})$, we get
\begin{align}\label{FEM-darcy}
a_p(\varphi_h,\psi_h)=\rho g(f_p,\psi_h)+\langle\bu_h\cdot\bn_f,\psi_h\rangle_{\Gamma}.
\end{align}
Applying a similar argument used for proving Lemma $\ref{Newton-sta-lemma}$, we decompose $\varphi_h$ as $\varphi_h=\phi^b_h+\phi^i_h$, where $\phi^b_h$ is the solution of $(\ref{aux-problem-darcy})$ with $\zeta=\bu_h\cdot\bn_f$, and $\phi^i_h$ is the solution of $(\ref{aux-problem-2})$ with $\eta=f_p$. Then, for the interface term, we have
\begin{align}\nonumber
\rho g \langle \varphi_h, \bu_h\cdot\bn_f\rangle_{\Gamma}
&=\rho g\langle \phi^b_h, \bu_h\cdot\bn_f\rangle_{\Gamma}+\rho g\langle \phi^i_h,\bu_h\cdot\bn_f\rangle_{\Gamma}\\\nonumber
&=a_p(\phi^b_h,\phi^b_h)+\rho g\langle \phi^i_h,\bu_h\cdot\bn_f\rangle_{\Gamma}\\\label{FEM-solu-est-2}
&\geq\rho g\langle \phi^i_h,\bu_h\cdot\bn_f\rangle_{\Gamma}.
\end{align}
The combination of $(\ref{FEM-solu-est-1})-(\ref{FEM-solu-est-2})$ implies the estimate $(\ref{FEM-solu-est})$.
\end{proof}

Finally, we derive error estimates of Newton iterative method. Let $(\bu_h,\varphi_h,p_h)$ be the solution of the scheme $(\ref{FEM-1})-(\ref{FEM-2})$ and $(\bu^h_n,p^h_n,\varphi^h_n)$ be the solution of Newton iterative method $(\ref{New-alg-1})-(\ref{New-alg-3})$. Set $\be^u_{n}=\bu_h-\bu^h_n,~e^{p}_n=p_h-p^h_n$ and $e^{\varphi}_n=\varphi_h-\varphi^h_n$.
 \begin{theorem}
 \label{convergence-rate}
 Under the assumptions of Lemma $\ref{Newton-sta-lemma}$, the following error estimates hold true.
 \begin{align}\label{Newton-err-est-1}
 \|\dd(\be^{u}_n)\|_{0,\Omega_f}&\leq \frac{C_1}{\nu}\left(\frac{3C_1\N_d}{2\nu^2}(\|\bm{f}_f\|_*+\|f_p\|_*)\right)^{2^n-1}(\|\bm{f}_f\|_*+\|f_p\|_*),\\\label{Newton-err-est-2}
|\nabla e^{\varphi}_n|_{1,\Omega_p}
&\leq  \frac{\tilde{C}_tC_tC_kC_1}{ \alpha_1\nu}
\left(\frac{3C_1\N_d}{2\nu^2}(\|\bm{f}_f\|_*+\|f_p\|_*)\right)^{2^n-1}(\|\bm{f}_f\|_*+\|f_p\|_*),\\\label{Newton-err-est-3}
\|e^p_n\|_{0,\Omega_f}
&\leq
(\frac{10}{3}+\frac{\tilde{C^2}_tC^2_tC^2_k}{\alpha_1\nu})\frac{C_1}{\beta}\\\nonumber
&~~~\left(\frac{3C_1\N_d}{2\nu^2}(\|\bm{f}_f\|_*+\|f_p\|_*)\right)^{2^{n}-1}(\|\bm{f}_f\|_*+\|f_p\|_*),
 \end{align}
 for all $n\geq 1$.
 \end{theorem}
 \begin{proof}
Subtracting $(\ref{New-alg-1})-(\ref{New-alg-3})$ from the finite element scheme $(\ref{FEM-1})-(\ref{FEM-2})$, we find the error equations for  Newton iterative method are determined by
\begin{align}\nonumber
&a_f(\be^{u}_n,\bv_h)+a_p(e^{\varphi}_n,\psi_h)+b(\bv_h,e^p_n)+a_{\Gamma}(\be^u_n,\psi_h;\bv_h, e^{\varphi}_n)+c(\be^u_n;\bu^h_{n-1},\bv_h)\\\label{err-equ-1}
&+c(\bu^h_{n-1};\be^u_h,\bv_h)+c(\be^u_{n-1};\be^u_{n-1},\bv_h)
 =0~~~~\forall~ \boldsymbol{v}_h\in \bs{X}_{f,h}, \psi_h \in X_{p,h},\\\label{err-equ-3}
&b(\be^u_n,q_h)=0~~~~\forall~q_h\in Q_h.
\end{align}
Taking $\bv_h=\be^u_n,~\psi_h=0$ in $(\ref{err-equ-1})$ and noting that $b(\be^u_n,e^p_n)=0$, we achieve
\begin{align*}
&2\nu \|\dd(\be^u_n)\|^2_{0,\Omega_f}\\
&=-\rho g\langle e^{\varphi}_n,\be^u_n\cdot\bn_f\rangle_{\Gamma}-c(\be^u_{n-1};\be^u_{n-1},\be^u_n)
-c(\be^u_n;\bu^h_{n-1},\be^u_n)-c(\bu^h_{n-1};\be^u_{n},\be^u_n)\\
&\leq \N_d\|\dd(\be^u_{n-1})\|^2_{0,\Omega_f}\|\dd(\be^u_n)\|_{0,\Omega_f}
+2\N_d\|\dd(\bu^h_{n-1})\|_{0,\Omega_f}\|\dd(\be^u_n)\|^2_{0,\Omega_f},
 \end{align*}
and with $(\ref{Newton-sta})$ and $(\ref{Newton-assume1})$, we further have
 \begin{align}\label{Newton-iter-est}
\|\dd(\be^u_n)\|_{0,\Omega_f}\leq \frac{3\N_d}{2\nu}\|\dd(\be^u_{n-1})\|^2_{0,\Omega_f},
 \end{align}
 for $n\geq 1$.

Now, we prove $(\ref{Newton-err-est-1})$ by mathematical induction. It follows from the initial condition $\eqref{ini-con1}$ that (\ref{Newton-err-est-1}) holds for $n=0$. Assuming that $(\ref{Newton-err-est-1})$ holds for $n=J$, we are going to prove that the estimate still holds for $n=J+1$. Applying the estimate $(\ref{Newton-iter-est})$ and the induction assumption gives
\begin{align*}
\|\dd(\be^u_{J+1})\|_{0,\Omega_f}&\leq \frac{3\N_d}{2\nu}\|\dd(\be^u_J)\|^2_{0,\Omega_f}\\
&\leq \frac{3\N_d}{2\nu}\cdot\frac{C^2_1}{\nu^2}\left(\frac{3C_1\N_d}{2\nu^2}(\|\bm{f}_f\|_*+\|f_p\|_*)\right)^{2^{J+1}-2}(\|\bm{f}_f\|_*+\|f_p\|_*)^2\\
&=\frac{C_1}{\nu}\left(\frac{3C_1\N_d}{2\nu^2}(\|\bm{f}_f\|_*+\|f_p\|_*)\right)^{2^{J+1}-1}(\|\bm{f}_f\|_*+\|f_p\|_*),
\end{align*}
which is $(\ref{Newton-err-est-1})$ with $n=J+1$. So we prove $(\ref{Newton-err-est-1})$ holds.

Let us consider the boundedness of $|e^{\varphi}_n|_1$. Taking $\bv_h=\mathbf{0},~\psi_h=e^{\varphi}_n$ in $(\ref{err-equ-3})$ and applying $(\ref{Newton-err-est-1})$ gives
\begin{align*}
\alpha_1|\nabla e^{\varphi}_n|^2_{1,\Omega_p}&\leq \|e^{\varphi}_n\|_{0,\Gamma}\|\be^u_n\|_{0,\Gamma}\\
&\leq  \tilde{C}_t|\nabla e^{\varphi}_n|_{1,\Omega_p}C_tC_k\|\dd(\be^u_n)\|_{0,\Omega_f}\\
&\leq  \frac{\tilde{C}_tC_tC_kC_1}{\nu}|\nabla e^{\varphi}_n|_{1,\Omega_p}
\left(\frac{3C_1\N_d}{2\nu^2}(\|\bm{f}_f\|_*+\|f_p\|_*)\right)^{2^n-1}(\|\bm{f}_f\|_*+\|f_p\|_*).
\end{align*}

Finally, we estimate $e^p_n$. By using the discrete inf-sup condition $\eqref{inf-sup}$, and the estimates $\eqref{Newton-sta}$, $(\ref{Newton-assume1})$, $\eqref{Newton-err-est-1}$ and $\eqref{Newton-err-est-2}$, we obtain{\color{black}
\begin{align*}
\beta\|e^p_n\|_{0,\Omega_f}
\leq& 2\nu\|\dd(\be^u_n)\|_{0,\Omega_f}+\rho g C_t\tilde{C}_{t}C_k|e^{\varphi}_n|_{1,\Omega_p}\\
&+2\N_d\|\dd(\bu^h_{n-1})\|_{0,\Omega_f}\|\dd(\be^u_n)\|_{0,\Omega_f}+\N_d\|\dd(\be^u_{n-1})\|^2_{0,\Omega_f}\\
\leq&(2\nu+\frac{\nu}{3})\|\dd(\be^u_n)\|_{0,\Omega_f}+\rho g C_t\tilde{C}_{t}C_k|e^{\varphi}_n|_{1,\Omega_p}
+\frac32\N_d\|\dd(\be^u_{n-1})\|^2_{0,\Omega_f}\\
\leq&\frac{7C_1}{3}\left(\frac{3C_1\N_d}{2\nu^2}(\|\bm{f}_f\|_*+\|f_p\|_*)\right)^{2^n-1}(\|\bm{f}_f\|_*+\|f_p\|_*)\\
&+\frac{\tilde{C}^2_tC^2_tC^2_kC_1}{ \alpha_1\nu}
\left(\frac{3C_1\N_d}{2\nu^2}(\|\bm{f}_f\|_*+\|f_p\|_*)\right)^{2^n-1}(\|\bm{f}_f\|_*+\|f_p\|_*)\\
&+{C_1}\left(\frac{3C_1\N_d}{2\nu^2}(\|\bm{f}_f\|_*+\|f_p\|_*)\right)^{2^{n}-1}(\|\bm{f}_f\|_*+\|f_p\|_*)\\
=&(\frac{10}{3}+\frac{\tilde{C}^2_tC^2_tC^2_k}{\alpha_1\nu}){C_1}
\left(\frac{3C_1\N_d}{2\nu^2}(\|\bm{f}_f\|_*+\|f_p\|_*)\right)^{2^{n}-1}(\|\bm{f}_f\|_*+\|f_p\|_*).
\end{align*}}
The proof is complete.
 \end{proof}

\begin{corollary}
The initial function $\bu^h_0\in \boldsymbol{X}_{f,h}$
defined by the following finite element method for Stokes Darcy problem:
\begin{align}\label{Newton-ini-1}
&a_f(\bu^h_0,\bv_h)+b(\bv_h,p^h_0)+\rho g\langle \varphi^h_0,\bv_h\cdot\bn_f\rangle_{\Gamma}=(\bm{f}_f,\bv_h)~~~~\forall~\bv_h\in \boldsymbol{X}_{f,h},\\\label{Newton-ini-2}
&b(\bu^h_0,q_h)=0~~~~\forall~q_h\in Q_h,\\\label{Newton-ini-3}
&a_p(\varphi^h_0,\psi_h)=\rho g\langle \psi_h,\bu^h_0\cdot\bn_f\rangle_{\Gamma}+\rho g(f_p,\psi_h)~~~~\forall~\psi_h\in X_{p,h},
\end{align}
satisfies the initial condition \eqref{ini-con1}-\eqref{ini-con2}.

\end{corollary}
\begin{proof}

Taking $\zeta=\bu^h_0\cdot\bn_f$ in auxiliary problem $\eqref{aux-problem-darcy}$ and $\eta=f_p$ in $(\ref{aux-problem-2})$, the corresponding solutions are denoted by $\phi^b_0$ and $\phi^i$, respectively. Then, we have $\phi^b_0=T_h(\bu^h_0\cdot\bn_f)$ and $\|\phi^i\|_{1,\Omega_p}\leq C_r\|f_p\|_*$. 
According to the superposition principle, we know that $\varphi^h_0=\phi^b_0+\phi^i$ solves $(\ref{Newton-ini-3})$, which implies that 
\begin{align}\label{tool-1}
\rho g \langle \varphi^h_0,\bv^h_0\cdot\bn_f\rangle_{\Gamma}&=\rho g\langle \phi^b_0, \bv^h_0\cdot\bn_f\rangle_{\Gamma}+\rho g\langle\phi^i, \bv^h_0\cdot\bn_f\rangle_{\Gamma}\\\nonumber
&=a_p(T_h(\bv^h_0\cdot\bn_f),\phi^b_0)+\rho g \langle\phi^i, \bv^h_0\cdot\bn_f\rangle_{\Gamma}.
\end{align}
Taking $\bv_h=\bu^h_0$ in $(\ref{Newton-ini-1})$, and using (\ref{tool-1}) and \eqref{aux-est-1}, we have
\begin{align}\label{stokes-darcy-sta}
2\nu\|\dd(\bu^h_0)\|^2_{0,\Omega_f}&\leq (\bm{f}_f,\bu^h_0)-\rho g \langle \varphi^h_0,\bu^h_0\cdot\bn_f\rangle_{\Gamma}\\\nonumber
&\leq (\bm{f}_f,\bu^h_0)-a_p(\phi^b_0,\phi^b_0)-\rho g\langle\phi^i, \bu^h_0\cdot\bn_f\rangle_{\Gamma}\\\nonumber
&\leq (\bm{f}_f,\bu^h_0)-\rho g \langle\phi^i, \bu^h_0\cdot\bn_f\rangle_{\Gamma}\\\nonumber
&\leq \|\bm{f}_f\|_*\|\dd(\bu^h_0)\|_{0,\Omega_f}+\rho g C_r\tilde{C}_{t}C_{t}C_k\|f_p\|_*\|\dd(\bu^h_0)\|_{0,\Omega_f}\\\nonumber
&\leq C_1(\|\bm{f}_f\|_*+\|f_p\|_*)\|\dd(\bu^h_0)\|_{0,\Omega_f},
\end{align}
which yields $(\ref{ini-con1})$.

Next, we verify the initial condition $(\ref{ini-con2})$. Let $\bu_h,~\varphi_h$ be the solution of finite element scheme $(\ref{FEM-1})-(\ref{FEM-2})$. For convenience, set $\be^u_0=\bu_h-\bu^0_h$, $e^{\varphi}_0=\varphi_h-\varphi^0_h$.

Subtracting $\eqref{Newton-ini-1}-\eqref{Newton-ini-3}$ from $\eqref{FEM-1}-\eqref{FEM-2}$ gives the error equations:
 \begin{align}\label{ini-err-equ-1}
 &a_f(\be^u_0,\bv_h)+a_p(e^{\varphi}_0,\psi_h)+a_{\Gamma}(\be^u_0,\psi_h;\bv_h,e^{\varphi}_0)+b(\bv_h,e^p_0)
 +c(\bu_h;\bu_h,\bv_h)=0,\\\label{ini-err-equ-2}
 &b(\be^u_0,q_h)=0,
 \end{align}
for any $\bv_h\in \boldsymbol{X}_{f,h},~\psi_h\in X_{p,h}$ and $q_h\in Q_h$.\\
Taking $\bv_h=\be^u_0,~\psi_h=0$ in $(\ref{ini-err-equ-1})$ and noting that $b(\be^u_0,e^p_0)=0$, we have 
\begin{align}\label{err-tool}
\nu\|\dd(\be^u_0)\|^2_{0,\Omega_f}= -\rho g\langle e^{\varphi}_0,\be^u_0\cdot\bn_f\rangle_{\Gamma}-c(\bu_h;\bu_h,\be^u_0).
\end{align}
Now, it suffices to estimate the terms on the right side of $\eqref{err-tool}$. Taking $\bv_h=\mathbf{0}$, $\psi_h=e^{\varphi}_0$ in $(\ref{ini-err-equ-1})$ yields that
\begin{align}\label{ini-err-sd1}
\rho g\langle\be^u_0\cdot\bn_f,e^{\varphi}_0\rangle_{\Gamma}=a_p(e^{\varphi}_0,e^{\varphi}_0)\geq 0.
\end{align}
Using the estimates $(\ref{tri-est})$ and $(\ref{FEM-solu-est})$ and the assumption $(\ref{Newton-assume1})$, we have
\begin{align}\nonumber
|c(\bu_h;\bu_h,\be^u_0)|&\leq \N_d\|\dd(\bu_h)\|^2_{0,\Omega_f}\|\dd(\be^u_0)\|_{0,\Omega_f}\\\nonumber
&\leq \frac{9\N_dC^2_1}{4\nu^2}\left(\|\bm{f}_f\|_*+\|f_p\|_*\right)^2\|\dd(\be^u_0)\|_{0,\Omega_f}\\\nonumber
&\leq \frac{9\N_dC^2_1}{4\nu^2}\cdot \frac{\nu^2}{9C_1\N_d}\left(\|\bm{f}_f\|_*+\|f_p\|_*\right)\|\dd(\be^u_0)\|_{0,\Omega_f}\\\label{ini-err-sd2}
&\leq \frac{1}{4}C_1\left(\|\bm{f}_f\|_*+\|f_p\|_*\right)\|\dd(\be^u_0)\|_{0,\Omega_f}.
\end{align}
By the estimates $\eqref{ini-err-sd1}$ and $\eqref{ini-err-sd2}$,
\begin{align*}
\|\dd(\be^u_0)\|_{0,\Omega_f}\leq \frac{C_1}{4\nu}\left(\|\bm{f}_f\|_*+\|f_p\|_*\right).
\end{align*}
which implies the initial condition $(\ref{ini-con2})$.

\end{proof}

\section{The Int-Deep method}\label{sec:algo}
In this section, we first design a PINN-type (cf. \cite{Karniadakis2019,Karniadakis2021}) deep learning method for the Navier-Stokes Darcy model and then propose the Int-Deep algorithm, where the solution obtained from the DL method serves as an initial function for Newton iterative method.
\subsection{Deep learning method}\label{sec:DL}
In a DL method, the solution of a PDE is approximated by deep neural networks functions. In this work, we employ the classical residual neural network (ResNet) \cite{2016-Resnet} to approximate  the solutions of the Navier-Stokes Darcy model. For this purpose, we first recall its structure. Given an input $\boldsymbol{x}\in\mathbb{R}^{d}$,
consider the following residual network architecture with skip connection in each layer.
$$
\begin{aligned}
\boldsymbol{h}_0 &=\boldsymbol{V} \boldsymbol{x}, \\
\boldsymbol{g}_l &=\sigma\left(\boldsymbol{W}_l \boldsymbol{h}_{l-1}+b_{l}\right),\;l=1, \ldots, L, \\
\boldsymbol{h}_l &=\boldsymbol{h}_{l-1}+ \boldsymbol{g}_l,\hspace{1.2cm} l=1, \ldots, L, \\
\boldsymbol{\phi}(\boldsymbol{x} ; \boldsymbol{\tilde{\theta}}) &=\boldsymbol{a}^T \boldsymbol{h}_L,
\end{aligned}
$$
where $\boldsymbol{\tilde{\theta}}=\{\boldsymbol{W}_l,\boldsymbol{b}_{l},\boldsymbol{a}\}, \boldsymbol{V} \in \mathbb{R}^{m \times d}$,
\[
V(i,j)=
\begin{cases}
1  &i=j,j\le d\\
0  &\mathrm{otherwise}
\end{cases}
,\; \boldsymbol{W}_l \in \mathbb{R}^{m \times m},\; \boldsymbol{b}_l\in\mathbb{R}^{m},
\]
$\boldsymbol{a} \in\mathbb{R}^{m\times n}$, $L$ is the number of layers, $m$ is the width of the residual blocks and $\sigma$ is an activation function.


Next, we present the DL method in detail. 
Following the approach in \cite{JK-2018-BC}, for the given essential boundary conditions $\bu=\boldsymbol{g}_{\bu}(\bx),\;\bx\in \Gamma_{f},\;\varphi(\bx)=g_{\varphi}(\bx),\;\bx\in \Gamma_{p}$,
we construct three neural network functions
$\boldsymbol{\phi}_{\bs{u}}(\bx;\btheta_{\bs{u}})$,
 $\phi_p(\bx;\btheta_p)$ and $\phi_{\varphi}(\bx;\btheta_{\varphi})$ as follows.
\begin{align*}
        \boldsymbol{\phi}_{\bs{u}}(\bx;\btheta_{\bs{u}})&=B_{\bs{u}}(\bx)\boldsymbol{N}_{\bs{u}}(\bx;\btheta_{\bs{u}})+\boldsymbol{\bar{g}}_{\bu}(\bx),\\
        \phi_{p}(\bx;\btheta_{p})&=N_{p}(\bx;\btheta_{p}),\\ \phi_{\varphi}(\bx;\btheta_{\varphi})&=B_{\varphi}(\bx)N_{\varphi}(\bx;\btheta_{\varphi})+\bar{g}_{\varphi}(\boldsymbol{x}),
\end{align*}
where $B_{\bs{u}}(\bx)$ and $B_{\varphi}(\boldsymbol{x})$ are two smooth scalar functions which satisfy the conditions $B_{\bs{u}}(\bx)|_{\Gamma_f}=0$ and $B_{\varphi}(\boldsymbol{x})|_{\Gamma_p}=0$; $\boldsymbol{N}_{\bs{u}}(\bx;\btheta_{\bs{u}})$, $N_p(\bx;\btheta_p)$ and $N_{\varphi}(\bx;\btheta_{\varphi})$ are three distinct residual neural network functions; $\boldsymbol{\bar{g}_{u}}$ is the extension of $\boldsymbol{g_{u}}$ to $\Omega_{f}$, and $\bar{g}_{\varphi}$ is the extension of $g_{\varphi}$ to $\Omega_{p}$. Obviously, all the above neural network functions automatically satisfy the prescribed boundary conditions.

Setting $\btheta:=\{\btheta_{\bs{u}},\btheta_{p},\btheta_{\varphi}\}$, we respectively approximate $\boldsymbol{u}$, $p$ and $\varphi$ by
\[
\bu(\bx)\approx\boldsymbol{\phi}_{\bs{u}}(\bx;\btheta_{\bs{u}}), \quad  p(\bx)\approx \phi_{p}(\bx;\btheta_{p}), \quad \varphi\approx \phi_{\varphi}(\bx;\btheta_{\varphi}),
\]
with the network parameters determined by minimizing the $\log$ square loss
\begin{equation}\label{eq:3.1}
    \boldsymbol{\theta^{*}}=\argmin_{\boldsymbol{\theta}}\log_{2}(I(\boldsymbol{\phi}_{\bs{u}}(\bx;\btheta_{\bs{u}}),\phi_{p}(\bx;\btheta_{p}),
    \phi_{\varphi}(\bx;\btheta_{\varphi}))+1),
\end{equation}

where
\begin{equation}\nonumber
\small
\begin{aligned}
&I(\boldsymbol{u},p,\varphi)\\
&=\gamma_{1} \left(|\Omega_{f}|\mathbb{E}_{\bxi} \left(\left\| -\operatorname{div}\bT\left(\bu(\bxi), p(\bxi)\right)+\rho\left(\bu(\bxi) \cdot \nabla\right) \bu(\bxi)-\boldsymbol{f}_{f}(\bxi) \right\|^{2}+ \left| \operatorname{div} \bu(\bxi)\right|^{2}\right) \right.\\
&\left.+|\Omega_{p}|\mathbb{E}_{\boldsymbol{\eta}} \left(\left|f_{p}(\boldsymbol{\eta})+\operatorname{div}(\boldsymbol{K} \nabla \varphi(\boldsymbol{\eta}))\right|^{2}\right) \right)\\
&+\gamma_{2}\left(|\Gamma| \mathbb{E}_{\boldsymbol{\zeta}} \left(\left|\bu(\boldsymbol{\zeta}) \cdot \boldsymbol{n_{f}} +\boldsymbol{K} \nabla \varphi(\boldsymbol{\zeta}) \cdot \boldsymbol{n}_{f}\right|^{2}\right.\right. \\
&+||\bs{K}||_{2}\left| \left(\rho g \varphi(\boldsymbol{\zeta}) +(\boldsymbol{T}(\bu(\boldsymbol{\zeta}),p(\boldsymbol{\zeta}))\boldsymbol{n}_{f})\cdot \boldsymbol{n}_{f})\right)\right|^{2}\\
&+\sum_{j=1}^{d-1}\Big(\big|(\boldsymbol{T}(\bu(\boldsymbol{\zeta}), p(\boldsymbol{\zeta}))\boldsymbol{n}_{f}) \cdot \boldsymbol{\tau}_{j}+\frac{\nu \alpha_{B J}}{\sqrt{\boldsymbol{\tau}_{j} \cdot \nu \boldsymbol{K} \cdot \boldsymbol{\tau}_{j}}} \bu(\boldsymbol{\zeta}) \cdot \boldsymbol{\tau}_{j}  \big|^{2}\Big)\bigg).
\end{aligned}
\end{equation}
Here $\gamma_{1}$, $\gamma_{2}$ are two positive parameters, $\|\cdot\|$ is the Euclidean norm of a vector, $\|\cdot\|_{2}$ is the 2-norm of a matrix, and $\boldsymbol{\xi}, \boldsymbol{\eta}, \boldsymbol{\zeta}$ are random variables that are uniformly distributed on $\Omega_{f},\;\Omega_{p}$ and $\Gamma$, respectively. Furthermore, we use the Monte-Carlo method to approximate the related expectation to derive the discretization of \eqref{eq:3.1}, which is solved by the stochastic gradient descent (SGD) method or its variants (e.g. Adam \cite{2014Adam}).

With these preparations, the DL algorithm for Navier-Stokes Darcy model $(\ref{model-pro1})-(\ref{model-pro4})$ is described as Algorithm \ref{alg:DL}. 
\vskip 0.5cm
\begin{algorithm}[H]
\footnotesize
	\caption{Deep learning method for Navier-Stokes Darcy problem.}
	\label{alg:DL}
	\textbf{Input}: the maximum number of training $Epoch$,
        the learning rate $\eta$,  sample size $N_{1},N_{2}$ and $ N_{3}$.\\
	\textbf{Output}:
	$$
	\bu^{DL}=\boldsymbol{\phi}_{\bs{u}}(\bx,\btheta_{\bs{u}}^{Epoch}),~~p^{DL}=\phi_{p}(\bx,\btheta_{p}^{Epoch}),~~
    \varphi^{DL}=\phi_{\varphi}(\bx,\btheta_{\varphi}^{Epoch}).
	$$
\begin{algorithmic}
		\STATE {\textbf{Initialization}: Let $l=0,~\eta^{0}=\eta$, {\color{black}the default initialization method in PyTorch is used to initialize the parameters of neural networks  $\btheta_{\bs{u}}^{0},~\btheta_{\bs{p}}^{0},~\btheta_{\bs{\varphi}}^{0}$.}}
		\WHILE {$l < Epoch$ }
		\STATE { {\color{black}Generate independent random variables  $\{\bs{\xi}_{i}\}_{i=1}^{N_{1}},\;\{\bs{\eta}_{i}\}_{i=1}^{N_{2}}$, and $\{\bs{\zeta}_{i}\}_{i=1}^{N_{3}}$ that are uniformly distributed on $\Omega_f$, $\Omega_{p}$, and $\Gamma$, respectively.} }
		\STATE {Computing loss:
		\begin{equation}\label{eq:3.2}
        L= \log_{2}(I_{1}(\boldsymbol{\phi}_{\bs{u}}(\btheta_{\bs{u}}^{l}),\phi_{p}(\btheta_{p}^{l}),\phi_{\varphi}(\btheta_{\varphi}^{l}))+1),
        \end{equation}
        where
        \begin{equation}\nonumber
        \begin{aligned}
        &I_{1}(\bu,p,\varphi)\\
        &=\gamma_{1} \left(\frac{|\Omega_{f}|}{N_1} \sum _{i=1}^{N_{1}} (\left\| -\operatorname{div}\boldsymbol{T}\left(\bu(\bxi_{i}), p(\bxi_{i})\right)+\rho\left(\bu(\bxi_{i}) \cdot \nabla\right) \bu(\bxi_{i})-\boldsymbol{f}_{f}(\bxi_{i}) \right\|^{2}\right.\\
        &\left.+ \left| \operatorname{div}\bu(\bxi_{i})\right|^{2}) +\frac{|\Omega_{p}|}{N_2} \sum _{i=1}^{N_{2}} \left(\left|f_{p}(\bta_{i})+\operatorname{div}(\boldsymbol{K} \nabla \varphi(\bta_{i}))\right|^{2}\right) \right)\\
        &+\gamma_{2}\bigg(\frac{|\Gamma|}{N_{3}} \sum _{i=1}^{N_{3}} (\big|\bu(\boldsymbol{\zeta}_{i}) \cdot \boldsymbol{n}_{f} +\boldsymbol{K} \nabla \varphi(\boldsymbol{\zeta}_{i}) \cdot \boldsymbol{n}_{f}\big|^{2}\\
        &+||\bs{K}||_{2}\big|\left(\rho g \varphi(\boldsymbol{\zeta}_{i}) +(\boldsymbol{T}(\bu(\boldsymbol{\zeta}_{i}),p(\boldsymbol{\zeta}_{i})) \boldsymbol{n}_{f})\cdot \boldsymbol{n}_{f})\right)\big|^{2}\\
        &+\sum_{j=1}^{d-1}\Big(\big|(\boldsymbol{T}(\bu(\boldsymbol{\zeta}_i), p(\boldsymbol{\zeta}_{i})) \boldsymbol{n}_{f}) \cdot \boldsymbol{\tau}_{j}+\frac{\nu \alpha_{B J}}{\sqrt{\boldsymbol{\tau}_{j} \cdot \nu \boldsymbol{K} \cdot \boldsymbol{\tau}_{j}}} \bu(\boldsymbol{\zeta}_{i}) \cdot \boldsymbol{\tau}_{j}  \big|^{2}\Big)\bigg).\\
        \end{aligned}
        \end{equation}}
        \vspace{-0.3cm}
        \STATE{Updating parameters $\btheta_{\bs{u}}^{l+1}= \btheta_{\bs{u}}^{l}-\eta^{l}
        \nabla_{\boldsymbol{\btheta_{\bs{u}}}}L,~ \btheta_{p}^{l+1}= \btheta_{p}^{l}- \eta^{l}
        \nabla_{\boldsymbol{\btheta_{p}}}L,~ \btheta_{\varphi}^{l+1}= \btheta_{\varphi}^{l}- \eta^{l}
        \nabla_{\boldsymbol{\btheta_{\varphi}}}L.$}
        \STATE{Let $l=l+1.$}
		\ENDWHILE
	\end{algorithmic}
\end{algorithm}

\subsection{The Int-Deep method for solving the Navier-Stokes Darcy problem}\label{sec:int-deep}
Since Newton iterative method is only locally convergent, it is very challenging to choose reasonably an initial guess. One natural choice is taking the initial guess as the solutions to the Stokes Darcy model (cf. \eqref{Newton-ini-1}-\eqref{Newton-ini-3}). However, the convergence is very sensitive to the physical parameters in the model. Motivated by the ideas in \cite{intdeep2020}, we are tempted to choose $I_h\bu^{DL}$ as an initial guess, where $\bu^{DL}$ denotes the deep learning solution obtained from Algorithm \ref{alg:DL} with few iteration steps, and $I_h$ is the usual nodal interpolation operator (cf. \cite{BrennerScott2008,Ciarlet1978}). Similar to \cite{intdeep2020}, this method is referred to as the Int-Deep method for the Navier-Stokes Darcy problem $(\ref{model-pro1})-(\ref{model-pro4})$, and is described as Algorithm \ref{alg:int-deep}.
It is worthy to emphasize that the later choice may rely partly on the so-called frequency principle mentioned in \cite{Xu-2020-frequency, Xu-2019-frequency}, as mentioned in the introduction part. Our numerical results also indicate the superiority of such a choice.

\begin{algorithm}[H]
\setstretch{1.45}
	\caption{Int-Deep method for Navier-Stokes Darcy problem.}
	\label{alg:int-deep}
	\textbf{Input}: the target accuracy $\epsilon$, the maximum number of iterations $N_{\max}$,
	the approximate solution in a form of a DNN $\bu^{DL}$.\\
	\textbf{Output}: $\boldsymbol{u}^h_{ID} = \boldsymbol{u}^h_{n},\;p^h_{ID} = p_{n}^{h},\;\varphi^h_{ID} = \varphi^h_{n}$.

\begin{algorithmic}
		\STATE {\textbf{Initialization}: Let $\boldsymbol{u}_0^h = {I_h \bu^{DL}}$, $n=1$, and $e_0=1$.}
		\WHILE {$e_{n-1} > \epsilon$ and $n \le N_{\max}$}
		\STATE {  find $\bu^h_n\in \mathbf{X}_{f,h},\;p^h_n\in Q_h$, and $\varphi^h_n\in X_p$ such that
        \begin{align}\nonumber
        &a_f(\bu^h_n,\bv_h)+b(\bv_h,p^h_n)+\rho g \langle\phi^h_n, \bv_h\cdot\bn_f\rangle_{\Gamma}+c(\bu^h_n;\bu^h_{n-1},\bv_h)
        +c(\bu^h_{n-1};\bu^h_n,\bv_h)\\ \nonumber
        &=(\bm{f}_f,\bv_h)+c(\bu^h_{n-1};\bu^h_{n-1},\bv_h),\\ \nonumber
        &b(\bu^h_n,q_h)=0,\\ \nonumber
        &a_p(\varphi^h_n,\psi_h)=\rho g \langle \psi_h,\bu^h_n\cdot\bn_f\rangle_{\Gamma}+\rho g(f_p,\psi_h).
        \end{align}}
        \vspace{-0.3cm}
		\STATE {$e_{n} = \max(\frac{\| \boldsymbol{u}_{n}^h - \boldsymbol{u}_{n-1}^h \|_{0,\Omega_f}}{\| \boldsymbol{u}_{n-1}^h \|_{0,\Omega_f}},
                              \frac{\| \varphi_{n}^h - \varphi_{n-1}^h \|_{0,\Omega_p}}{\| \varphi_{n-1}^h \|_{0,\Omega_p}},
                              \frac{\| p_{n}^h - p_{n-1}^h \|_{0,\Omega_f}}{\| p_{n-1}^h \|_{0,\Omega_f}}$), $n = n+1$.}
		\ENDWHILE
	\end{algorithmic}
\end{algorithm}

\section{Numerical experiments}\label{sec:test}

In this section, we shall present some numerical experiments to illustrate the performance of Algorithm \ref{alg:DL} and Algorithm \ref{alg:int-deep} for Navier-Stokes Darcy problem.

The specific details about the implementation of algorithms are provided in advance. In  Algorithm \ref{alg:DL}, we use ResNets of width 50, depth 5 for the Navier-Stokes problem and width 50, depth 6 for the Darcy problem. the activation function in all neural networks is taken as $\sigma=\max\{x^3,~0\}$. During the training process, we train the neural networks using the Adam optimizer \cite{2014Adam} with a learning rate of $\eta = 5\text{e}-03$. The learning rate decays exponentially and the decay rate is $0.01^\frac{1}{10000}$. The batch size is taken as: $N_1=N_2=1024$, $N_3=256$ for all 2D examples; $N_1=N_2=3000$, $N_3=512$ for the 3D example.
In Algorithm \ref{alg:int-deep}, we set $N_{\max}=20$, $\epsilon =1e-7$ for all examples. 

Some notations in the following part are summarized here. In numerical experiments, different strategies are employed to generate initial guesses for Newton iterative method, and corresponding results will be compared with the Int-Deep algorithm. We will refer to $\bs{u}^{DL}_{0}$ as the initial guess $I_{h}\bs{u}^{DL}$ where $\bs{u}^{DL}$ is provided by Algorithm \ref{alg:DL} with 200 epochs in the Adam, $\boldsymbol{u}^{SD}_{0}$ as the initial guess obtained by solving the discretized Stokes Darcy problem \eqref{Newton-ini-1} -\eqref{Newton-ini-3}, and $\boldsymbol{C}$ means a finite element function whose vector components under the nodal basis functions are taken as a constant $C$. The number of epochs in the Adam for DL method is denoted by $\#$Epoch and the number of iterations in Newton iterative method is denoted by $\#\mathrm{K}$. 
The accuracy of the solution of Algorithm $\ref{alg:DL}$ is measured by the discrete maximum norm: for any $v\in C(\bar{\Omega})$,
$$
\|v\|_{0,\infty,h}=\max_{x\in \Omega_h}|\bv(\bx)|,
$$
where $\Omega_h$ is the set of all vertices in $\mathcal{T}_h.$
We measure the accuracy of solution of Algorithm \ref{alg:int-deep} by the relative $L^2$-error and $H^1$-error:
$$
\begin{array}{llll}
\|\boldsymbol{e}_{\bs{u}}^h\|_0=\frac{\left\|\boldsymbol{u}-\boldsymbol{u}^h_{ID}\right\|_{0, \Omega_f}}{\left\|\boldsymbol{u}\right\|_{0, \Omega_f}},  &
\|e_p^h\|_0=\frac{\left\|p-p^h_{ID}\right\|_{0, \Omega_f}}{\left\|p\right\|_{0, \Omega_f}}, &
\|e^{\varphi_h}_h\|_0=\frac{\left\|\varphi_h-\varphi^h_{ID}\right\|_{0, \Omega_p}}{\left\|\varphi\right\|_{0, \Omega_p}}, \\
\|\boldsymbol{e}_{\bs{u}}^h\|_1=\frac{\left|\boldsymbol{u}-\boldsymbol{u}^h_{ID}\right|_{1, \Omega_f}}{\left|\boldsymbol{u}\right|_{1, \Omega_f}},&
\|e^h_{\varphi}\|_1=\frac{\left|\varphi-\varphi^h_{ID}\right|_{1, \Omega_p}}{\left|\varphi\right|_{1, \Omega_p}}.&
\end{array}
$$

\subsection{A 2D--example with a closed-form solution}
 In this example, we consider the coupled problem $(\ref{model-pro1})-(\ref{model-pro4})$ in $\Omega\in \mathbb{R}^2$ with $\Omega_{f}=(0,\pi) \times(0,\pi),\;\Omega_{p}=(0,\pi) \times$ $(-\pi,0)$ and the interface $\Gamma=(0,\pi) \times\{0\}$. 
 The exact solution is given by 
\begin{equation*}
  \begin{cases}
        u=2 \sin y \cos y \cos x  &\text { in } \Omega_{f},  \\
        v=\left(\sin ^{2} y-2\right) \sin x &\text { in } \Omega_{f} ,\\
        p=\sin x \sin y + \frac{1}{3\kappa}, & \text { in } \Omega_{f},\\
        \varphi =\frac{1}{\kappa}(\left(e^{y}-e^{-y}\right) \sin x+\frac{1}{3}) &\text { in } \Omega_{p},
  \end{cases}
\end{equation*}
where the components of $\boldsymbol{u}$ are denoted by $(u, v)$ for convenience. For simplicity, all the parameters except $\nu$ and $\kappa~ (\boldsymbol{K}=\kappa\boldsymbol{I}) $ in the coupled model are set to
be 1. In \eqref{eq:3.2}, we set $\gamma_{1}=200,\;\gamma_{2} = 1$; $B_{\bs{u}}(\bx)= x(x-\pi)(y-\pi)$ and $B_{\varphi}(\boldsymbol{x})= x(x-\pi)(y+\pi)$.  
To evaluate the test error of DL method, we adopt a mesh size $h=\frac{1}{1024}$ to generate the uniform triangulation $\mathcal{T}_h$ as the test locations.

\subsubsection{The Performance of Algorithm \ref{alg:DL}}
In this part, we focus on the  performance of Algorithm \ref{alg:DL} and  the results are presented in Table $\ref{tab:DLnu=1K=1} - \ref{tab:DLnu=1e-4K=1e-8}$. And we denote the components of $\bs{u}^{DL}$ obtained by Algorithm \ref{alg:DL} as $(u^{DL},v^{DL})$. Here are some observations.

\begin{enumerate}
    \item As the viscosity coefficient $\nu$ decreases, the absolute discrete maximum errors of all unknowns can reach $O(10^{-3})$ after 10000 epochs but it is difficult to further improve the accuracy with the increase of the iterations. When both the viscosity coefficient $\nu$ and permeability $\kappa$ are set to be extremely small, the errors of pressure $p$ and hydraulic head $\varphi$ can only reach $O(10^{-1})$, and the error of velocity $\boldsymbol{u}$ can still reach $O(10^{-3})$ after 10000 epochs. 

    \item For different values of viscosity coefficient $\nu$ and conductivity $\kappa$, the absolute discrete maximum errors of velocity $\boldsymbol{u}$ can reach $O(10^{-1})$ after 200 epochs. Therefore, the approximate velocity $\boldsymbol{u}^{DL}$ generate by Algorithm \ref{alg:DL} with $O(100)$ iterations can serve as a good initial guess for Algorithm \ref{alg:int-deep}. 
\end{enumerate}

\begin{table}[H]\scriptsize
     \centering
\setlength{\tabcolsep}{2.5pt}
\begin{tabular}{ccccc}
\hline
$\#$Epoch & $||u-u^{DL}||_{0,\infty,h}$ & $||v-v^{DL}||_{0,\infty,h}$ & $||\kappa(p-p^{DL})||_{0,\infty,h}$ & $||\kappa(\varphi-\varphi^{DL})||_{0,\infty,h}$  \\
\hline
200 & 1.3646e-01 & 3.6875e-01 & 2.5902e+00 & 4.3976e+00 \\
500 & 5.3574e-02 & 9.2439e-02 & 1.6767e+00 & 1.0690e+00 \\
1000 & 1.0403e-01 & 7.3808e-02 & 8.9771e-01 & 3.1771e-01 \\
1500 & 1.3238e-02 & 5.1877e-02 & 4.8701e-01 & 3.9559e-01 \\
2000 & 9.5502e-02 & 4.3894e-02 & 1.5858e-01 & 1.7449e-01 \\
2500 & 1.9290e-02 & 1.9688e-02 & 1.0117e-01 & 1.3614e-01 \\
3000 & 9.1979e-03 & 6.9558e-03 & 3.8412e-02 & 3.8035e-02 \\
3500 & 3.2457e-03 & 3.6885e-03 & 8.9594e-03 & 1.3716e-02 \\
5000 & 2.3976e-03 & 9.9764e-04 & 3.3161e-03 & 6.7926e-03 \\
7500 & 8.6575e-04 & 6.2326e-04 & 2.1703e-03 & 3.2009e-03 \\
10000 & 6.5700e-04 & 6.2362e-04 & 1.8386e-03 & 1.5161e-03 \\
15000 & 7.5354e-04 & 6.0582e-04 & 1.3077e-03 & 1.5129e-03 \\
20000 & 5.7673e-04 & 5.7690e-04 & 1.3107e-03 & 1.2378e-03 \\
25000 & 6.2294e-04 & 5.3808e-04 & 1.4011e-03 & 1.3049e-03 \\
30000 & 6.0878e-04 & 5.2679e-04 & 1.3369e-03 & 1.2133e-03 \\
\hline
\end{tabular}
    \caption{ The absolute discrete maximum errors of Algorithm\ref{alg:DL} with $\nu=1,\kappa=1$.}
     \label{tab:DLnu=1K=1}

\begin{tabular}{ccccc}
\hline
$\#$Epoch & $||u-u^{DL}||_{0,\infty,h}$ & $||v-v^{DL}||_{0,\infty,h}$ & $||\kappa(p-p^{DL})||_{0,\infty,h}$ & $||\kappa(\varphi-\varphi^{DL})||_{0,\infty,h}$   \\
\hline
200 & 1.5774e-01 & 1.6119e-01 & 3.2131e+00 & 3.6156e+00 \\
500 & 6.8583e-02 & 1.3629e-01 & 1.5524e+00 & 1.0916e+00 \\
1000 & 8.8662e-02 & 1.1680e-01 & 2.8396e-01 & 3.2233e-01 \\
1500 & 2.2016e-02 & 1.5472e-02 & 1.0725e-01 & 1.0995e-01 \\
2000 & 3.0242e-02 & 8.9225e-03 & 2.4742e-02 & 4.3176e-02 \\
2500 & 6.2917e-03 & 9.6236e-03 & 1.7075e-02 & 3.9242e-02 \\
3000 & 2.4351e-03 & 1.4384e-03 & 2.6748e-03 & 1.6315e-02 \\
3500 & 2.0461e-03 & 1.0641e-03 & 5.2479e-03 & 1.4417e-02 \\
5000 & 1.2189e-03 & 1.3845e-03 & 2.6431e-03 & 4.9044e-03 \\
7500 & 6.3689e-04 & 8.2492e-04 & 2.5790e-03 & 3.7786e-03 \\
10000 & 7.1952e-04 & 5.5685e-04 & 1.2228e-03 & 3.0202e-03 \\
15000 & 3.9990e-04 & 4.8316e-04 & 7.3376e-04 & 2.9747e-03 \\
20000 & 3.6191e-04 & 4.8170e-04 & 8.2649e-04 & 2.6312e-03 \\
25000 & 3.6458e-04 & 5.1755e-04 & 8.9596e-04 & 2.6613e-03 \\
30000 & 3.6652e-04 & 4.9583e-04 & 8.8039e-04 & 2.6657e-03 \\
\hline
\end{tabular}
    \caption{ The absolute discrete maximum errors of Algorithm\ref{alg:DL} with $\nu=0.01,\kappa=1$.}
     \label{tab:DLnu=0.01K=1}

\begin{tabular}{ccccc}
\hline
$\#$Epoch & $||u-u^{DL}||_{0,\infty,h}$ & $||v-v^{DL}||_{0,\infty,h}$ & $||\kappa(p-p^{DL})||_{0,\infty,h}$ & $||\kappa(\varphi-\varphi^{DL})||_{0,\infty,h}$  \\
\hline
200 & 7.3931e-02 & 9.8672e-02 & 2.1117e+00 & 2.5451e+00 \\
500 & 6.2953e-02 & 3.7513e-02 & 1.5959e+00 & 1.0343e+00 \\
1000 & 5.1847e-02 & 6.2577e-02 & 7.8331e-01 & 4.5935e-01 \\
1500 & 2.6548e-02 & 1.3333e-02 & 2.3349e-01 & 2.2663e-01 \\
2000 & 3.9128e-02 & 2.1147e-02 & 9.7124e-02 & 4.4471e-02 \\
2500 & 1.6828e-02 & 6.5111e-03 & 2.9291e-02 & 2.1896e-02 \\
3000 & 7.0802e-03 & 3.0962e-03 & 1.7970e-02 & 1.3612e-02 \\
3500 & 9.6858e-04 & 1.6541e-03 & 4.0128e-03 & 8.7207e-03 \\
5000 & 1.9639e-03 & 1.0026e-03 & 1.6727e-03 & 2.1226e-03 \\
7500 & 1.0221e-03 & 7.9086e-04 & 1.0696e-03 & 2.0161e-03 \\
10000 & 6.5865e-04 & 6.6323e-04 & 9.8430e-04 & 1.1517e-03 \\
15000 & 3.9810e-04 & 5.5374e-04 & 3.6499e-04 & 8.2533e-04 \\
20000 & 3.8367e-04 & 5.3969e-04 & 3.0734e-04 & 7.5236e-04 \\
25000 & 4.1257e-04 & 5.8360e-04 & 2.7854e-04 & 8.0696e-04 \\
30000 & 4.0072e-04 & 5.7836e-04 & 2.7543e-04 & 7.7025e-04 \\
\hline
\end{tabular}
    \caption{ The absolute discrete maximum errors of Algorithm\ref{alg:DL}  with $\nu=1e-04,\kappa=1$.}
     \label{tab:DLnu=1e-4K=1}
\begin{tabular}{ccccc}
\hline
$\#$Epoch & $||u-u^{DL}||_{0,\infty,h}$ & $||v-v^{DL}||_{0,\infty,h}$ & $||\kappa(p-p^{DL})||_{0,\infty,h}$ & $||\kappa(\varphi-\varphi^{DL})||_{0,\infty,h}$   \\
\hline
200 & 2.4375e-01 & 1.8926e-01 & 3.3333e-01 & 2.3097e+01 \\
500 & 1.8031e-01 & 1.7551e-01 & 3.3333e-01 & 2.0465e+01 \\
1000 & 1.1470e-01 & 2.1357e-01 & 3.3333e-01 & 2.8210e+00 \\
1500 & 8.0150e-02 & 1.1450e-01 & 3.3333e-01 & 1.1150e+00 \\
2000 & 2.4448e-02 & 6.2345e-02 & 3.3333e-01 & 6.0206e-01 \\
2500 & 5.8521e-02 & 6.4980e-02 & 3.3333e-01 & 7.5758e-01 \\
3000 & 1.2664e-02 & 2.8829e-02 & 3.3333e-01 & 4.6555e-01 \\
3500 & 1.2212e-02 & 7.3794e-03 & 3.3333e-01 & 2.7240e-01 \\
5000 & 2.0266e-03 & 3.0949e-03 & 3.3333e-01 & 2.0827e-01 \\
7500 & 1.3459e-03 & 3.1860e-03 & 3.3333e-01 & 1.9935e-01 \\
10000 & 9.9657e-04 & 2.5882e-03 & 3.3333e-01 & 1.9982e-01 \\
15000 & 9.3584e-04 & 2.8319e-03 & 3.3333e-01 & 2.0091e-01 \\
20000 & 9.1865e-04 & 2.7231e-03 & 3.3333e-01 & 1.9955e-01 \\
25000 & 9.0326e-04 & 2.7280e-03 & 3.3333e-01 & 1.9959e-01 \\
30000 & 9.0204e-04 & 2.7455e-03 & 3.3333e-01 & 1.9964e-01 \\
\hline
\end{tabular}
\caption{  The absolute discrete maximum errors of Algorithm\ref{alg:DL}  with $\nu=1e-04,\kappa=1e-08$.}
 \label{tab:DLnu=1e-4K=1e-8}
 \end{table}

\subsubsection{The performance of Algorithm \ref{alg:int-deep}}
In this part, we pay attention to the stability and effectiveness of the Algorithm \ref{alg:int-deep}. 
To this end, we first compare the iteration steps $\#\mathrm{K}$ of Newton iterative method employing different initial guesses with
$h = \frac{\pi}{2^k}$, $k=6, 7, 8$. Tables $\ref{tab:iter-nu=1K=1}-\ref{tab:iter-nu=1e-4K=1e-8}$ show that as the parameters $\nu$ and $\kappa$ decrease, the iterative steps of Algorithm \ref{alg:int-deep} remains almost unchanged, while Newton iterative method with other initial guesses need more iterative steps or even fail to convergence. Furthermore, when the parameter $\nu$ or $\kappa$ is extremely small, only Algorithm \ref{alg:int-deep} is capable of convergence.
\begin{table}[H]
     \centering
     \begin{tabular}{c|cccc}\hline
           \diagbox[width=12em]{Initial value}{Mesh size}  & $h=\frac{\pi}{64}$ &$h=\frac{\pi}{128}$& $h=\frac{\pi}{256}$\\\hline
        $\boldsymbol{u}_{0}^{SD}$   &4   &4  &4\\
         $\boldsymbol{0}$         &5   &5  &5 \\
        $\boldsymbol{1}$         &6   &6  &6\\
        $\boldsymbol{u}_{0}^{DL}$ &5   &5  &5\\\hline
     \end{tabular}
    \caption{The number of iterations $\#$K  under different initial guesses with $\nu=1,\kappa=1$.}
     \label{tab:iter-nu=1K=1}
 \end{table}
 \begin{table}[H]
     \centering
     \begin{tabular}{c|cccc}\hline
             \diagbox[width=12em]{Initial value}{Mesh size}  & $h=\frac{\pi}{64}$ &$h=\frac{\pi}{128}$& $h=\frac{\pi}{256}$\\\hline
                $\boldsymbol{u}_{0}^{SD}$  & $\times$  &$\times$  &$\times$ \\
               $\boldsymbol{0}$       & $\times$  &$\times$  &$\times$   \\
                 $\boldsymbol{1}$        &8   &8  &8 \\
               $\boldsymbol{u}_{0}^{DL}$ & 5         &5         &5    \\\hline
     \end{tabular}
    \caption{The number of iterations $\#$K  under different initial guesses  with $\nu=0.01,\;\kappa=1$.}
     \label{tab:iter-nu=0.01K=1}
 \end{table}
\begin{table}[H]
     \centering

     \begin{tabular}{c|cccc}\hline
            \diagbox[width=12em]{Initial value}{Mesh size}  & $h=\frac{\pi}{64}$ &$h=\frac{\pi}{128}$& $h=\frac{\pi}{256}$\\\hline
                 $\boldsymbol{u}_{0}^{SD}$  & $\times$  &$\times$  &$\times$ \\
               $\boldsymbol{0}$       & $\times$  &$\times$  &$\times$   \\
                 $\boldsymbol{1}$        &$\times$   &$\times$  &$\times$ \\
               $\boldsymbol{u}_{0}^{DL}$ & 5         &5         &5    \\\hline
     \end{tabular}
    \caption{The number of iterations $\#$K  under different initial guesses  with $\nu=1e-04,\;\kappa=1$.}
     \label{tab:iter-nu=1e-4K=1}
 \end{table}
\begin{table}[H]
     \centering
     \begin{tabular}{c|cccc}\hline
          \diagbox[width=12em]{Initial value}{Mesh size} & $h=\frac{\pi}{64}$& $h=\frac{\pi}{128}$ & $h=\frac{\pi}{256}$ \\\hline
            $\boldsymbol{u}_{0}^{SD}$   & $\times$  & $\times$ &$\times$    \\
           $\boldsymbol{0}$       & $\times$  &$\times$  &$\times$     \\
            $\boldsymbol{1}$       & $\times$  &$\times$  &$\times$     \\
            $\boldsymbol{u}_{0}^{DL}$ & 6         & 5        &5         \\\hline
     \end{tabular}
    \caption{The number of iterations $\#$K  under different initial guesses  with $\nu=1e-04,\;\kappa=1e-08$.}
     \label{tab:iter-nu=1e-4K=1e-8}
 \end{table}
\begin{remark}
   "$\times$" indicates that the iterative method does not converge within $N_{\max}=20$ steps.
\end{remark}

Next, we investigate the accuracy of Algorithm \ref{alg:int-deep} for different values of viscosity coefficient $\nu$ and hydraulic conductivity $\kappa$. 
The relative errors for $\boldsymbol{u}, \;p $ and $\varphi$ in $L^2$ and $H^1$ norms are displayed in Table \ref{tab:Nw-nu=1K=1}$-$Table \ref{tab:Nw-nu=1e-4K=1e-8}. Here are our observations.
\begin{table}[H]\scriptsize
\setlength{\tabcolsep}{2pt}
\centering
     \begin{tabular}{ccccccccccc}\hline
        &  $\|\boldsymbol{e}_{\bs{u}}^h\|_0$ &order & $\|e_{p}^h\|_0$ &order  & $\|e^h_{\varphi}\|_0$  &order  & $\|\boldsymbol{e}_{\bs{u}}^h\|_1$ &order  &  $\|e^h_{\varphi}\|_1$    &order                  \\\hline
         h$=\frac{\pi}{64}$  & 3.4269e-06&-  & 1.2657e-04 &-    & 1.6366e-06 &-& 4.2015e-04 &- &2.0316e-04 &- \\
         h$=\frac{\pi}{128}$ & 4.2851e-07& 2.9995  & 3.1603e-05 &2.0019 & 2.0458e-07 &    3.0000 & 1.0509e-04 &1.9993 & 5.0796e-05 &    1.9998\\
         h$=\frac{\pi}{256}$ & 5.3574e-08& 2.9997  & 7.8983e-06 &2.0004& 2.5573e-08 &    3.0000 &2.6279e-05  &    1.9997 &1.2700e-05 &    1.9999\\\hline
    \end{tabular}
    \caption{The convergence performance of Algorithm \ref{alg:int-deep} with $\nu=1,\kappa=1$. }
    \label{tab:Nw-nu=1K=1}
     \begin{tabular}{ccccccccccc}\hline
         &  $\|\boldsymbol{e}_{\bs{u}}^h\|_0$ &order & $\|e_{p}^h\|_0$ &order  & $\|e^h_{\varphi}\|_0$  &order  & $\|\boldsymbol{e}_{\bs{u}}^h\|_1$ &order  &  $\|e^h_{\varphi}\|_1$    &order                  \\\hline
         h$=\frac{\pi}{64}$ & 5.3721e-06  &-      &  1.2654e-04 &-      & 1.6367e-06 &-      & 5.2354e-04   &-  &2.0316e-04  &-\\
         h$=\frac{\pi}{128}$ & 5.0966e-07 &3.3979 &  3.1596e-05 &2.0017 & 2.0458e-07 &3.0000  & 1.1469e-04  &2.1905& 5.0796e-05&- \\
         h$=\frac{\pi}{256}$ & 5.6921e-08 &3.1625 &  7.8978e-06 &2.0002 & 2.5572e-08 &3.0000   & 2.7229e-05 &2.0746   & 1.2700e-05 &1.9998\\\hline
    \end{tabular}
    \caption{The convergence performance of Algorithm \ref{alg:int-deep} with $\nu=0.01,\kappa=1$. }
\centering

     \begin{tabular}{ccccccccccc}\hline
       &  $\|\boldsymbol{e}_{\bs{u}}^h\|_0$ &order & $\|e_{p}^h\|_0$ &order  & $\|e^h_{\varphi}\|_0$  &order  & $\|\boldsymbol{e}_{\bs{u}}^h\|_1$ &order  &  $\|e^h_{\varphi}\|_1$    &order                  \\\hline
         h$=\frac{\pi}{64}$ & 1.3970e-04  &-       &  1.7111e-04 &-       & 1.9279e-06&- & 1.3271e-02 &-  &2.0317e-04 &- \\
         h$=\frac{\pi}{128}$ & 1.6294e-05 &3.0999  &  3.6495e-05 &2.2292  & 2.3649e-07&3.0272  & 2.6950e-03&2.3000   & 5.0797e-05&1.9999 \\
         h$=\frac{\pi}{256}$ & 1.5417e-06 &3.4018  &  8.0830e-06  &2.1747& 2.7748e-08 &3.0913  & 4.8204e-04 &2.4831   & 1.2700e-05 &2.0000\\\hline
    \end{tabular}
    \caption{The convergence performance of Algorithm \ref{alg:int-deep} with $\nu=1e-04,\kappa=1$. }

    \label{tab:Nw-nu=1e-4K=1}
     \begin{tabular}{ccccccccccc}\hline
       &  $\|\boldsymbol{e}_{\bs{u}}^h\|_0$ &order & $\|e_{p}^h\|_0$ &order  & $\|e^h_{\varphi}\|_0$  &order  & $\|\boldsymbol{e}_{\bs{u}}^h\|_1$ &order  &  $\|e^h_{\varphi}\|_1$    &order                  \\\hline
         h$=\frac{\pi}{64}$ & 1.7757e-04&-  &  2.4257e-07 &-& 1.6367e-06  &- & 2.1475e-02 &- &2.0317e-04&-  \\
         h$=\frac{\pi}{128}$ & 2.0852e-05 &3.0901 &  1.3720e-08 &4.1440& 2.0458e-07 & 3.0000  & 4.9362e-03 &2.1212  & 5.0797e-05&1.9998 \\
         h$=\frac{\pi}{256}$ &2.1976e-06& 3.2462  &  7.7755e-10 &4.1412& 2.5573e-08 & 3.0000  & 1.1164e-03  &2.1444  & 1.2700e-05&2.0000 \\\hline
    \end{tabular}
    \caption{The convergence performance  of Algorithm \ref{alg:int-deep} with $\nu=1e-04,\kappa=1e-08$. }
    \label{tab:Nw-nu=1e-4K=1e-8}
\end{table}

\begin{enumerate}
    \item Regarding the Navier-Stokes part, as anticipated, the relative errors for velocity $\bs{u}$  in $L^{2}$ norm are of order $O(h^3)$, and the relative errors for $\bs{u}$  in the $H^1$ norm are of order $O(h^2)$ under various values of $ \nu $ and $ \kappa $.
    \item When $\kappa=1$, the order of relative $L^{2}$ norm errors for p is $O(h^2)$, which is as expected. However, when $\kappa$ is small, the order of relative errors is $O(h^4)$, which is unexpected.
    \item For the Darcy part, we observe that the relative errors for $\varphi$ in $L^{2}$ norm are of order $O(h^3)$ and the relative errors of $\varphi$ in the $H^1$ norm are of order $O(h^2)$, which is in line with our expectations. These results hold for different values of $ \nu $ and $\kappa$.
\end{enumerate}

\subsection{A 2D-example without closed form solution}
In this subsection, we test the Algorithm $\ref{alg:int-deep}$ for the example constructed in \cite{2021-ex2} and we make a slight adjustments here.

In the free fluid region $\Omega_{f}=(0,2) \times(0,1)$, the lid-driven problem is considered. 
The fluid is mainly driven by a rightward velocity ($\bs{u}=[1,0]^T$) on the top side $(y=1)$, and no-slip boundary condition $(\boldsymbol{u}=\bs{0})$ is imposed on the left- and right-sides. There is no body force $(\bs{f}_f=\bs{0})$. We take the parameter $\nu=1,\;0.1,\; 0.01$.

In the porous media region $\Omega_{p}=(0,2) \times(-1,0)$.
A heterogeneous permeability $\boldsymbol{K}=\kappa \boldsymbol{I}$ is taken. In three blocks: $(0.2,0.6)\times(-0.8,-0.6),(0.8,1.2)\times(-0.7,-0.5)$ and $(1.4,1.8)\times(-0.6,-0.4)$, the permeability $\kappa$ is set to be $10^{-6}$. For the remaining region, $\kappa=1$. There is no source ($f_{p}=0$). And $\varphi=0$ is imposed on the left-, right-, and bottom-sides of the region.

For simplicity, all the parameters except $\nu$ and $\kappa$ in the coupled model are set to 1. Besides, in \eqref{eq:3.2}, we take $\gamma_{1}=200, \gamma_{2} = 1, B_{\bs{u}}(\bx)= x(x-2)(y-1)$ and $B_{\varphi}(x)= x(x-2)(y+1)$.

Since there is no exact solution for comparison, we focus on corresponding physical phenomena. From Figure \ref{fig:ex3}, we can clearly observe that there is a smooth exchange of flow between the free flow (Navier-Stokes) and porous-medium flow (Darcy) across the interface ($y=0$). In Navier-Stokes region, singular pressure occurs at the two corners $(0,1)$ and $(2,1)$  and in Darcy region, the flow path is diverted around three low-permeability blocks.
With the vary of $\nu$, the fluid exhibits distinct flow patterns. For $\nu=1$ and $\nu=0.1$, fluid travels from the Navier-Stokes region to the  Darcy region when $x >1$, and from the Darcy region to the Navier-Stokes region when $x<1$. For $\nu=0.01$, fluid travels the Navier-Stokes region to the Darcy region when $x>1.3$, and from the Darcy domain to the Navier-Stokes domain when $x<1.3$.
\begin{figure}[H]
    \begin{minipage}[c]{0.49\linewidth}
         \centering
        \includegraphics[scale=0.22]{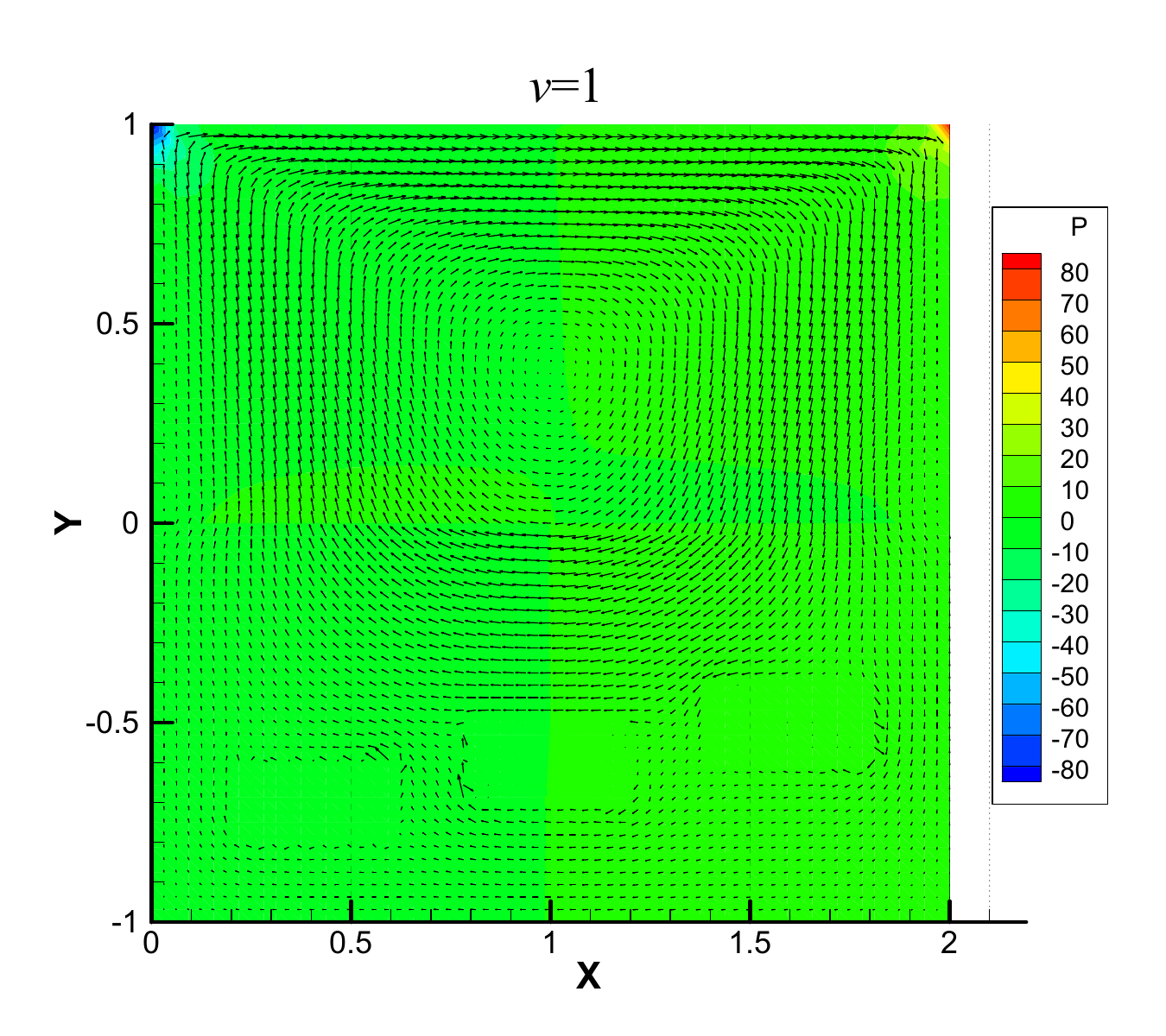}
        \label{ex5:nu=1DL}
    \end{minipage}
    \begin{minipage}[c]{0.49\linewidth}
         \centering
        \includegraphics[scale=0.22]{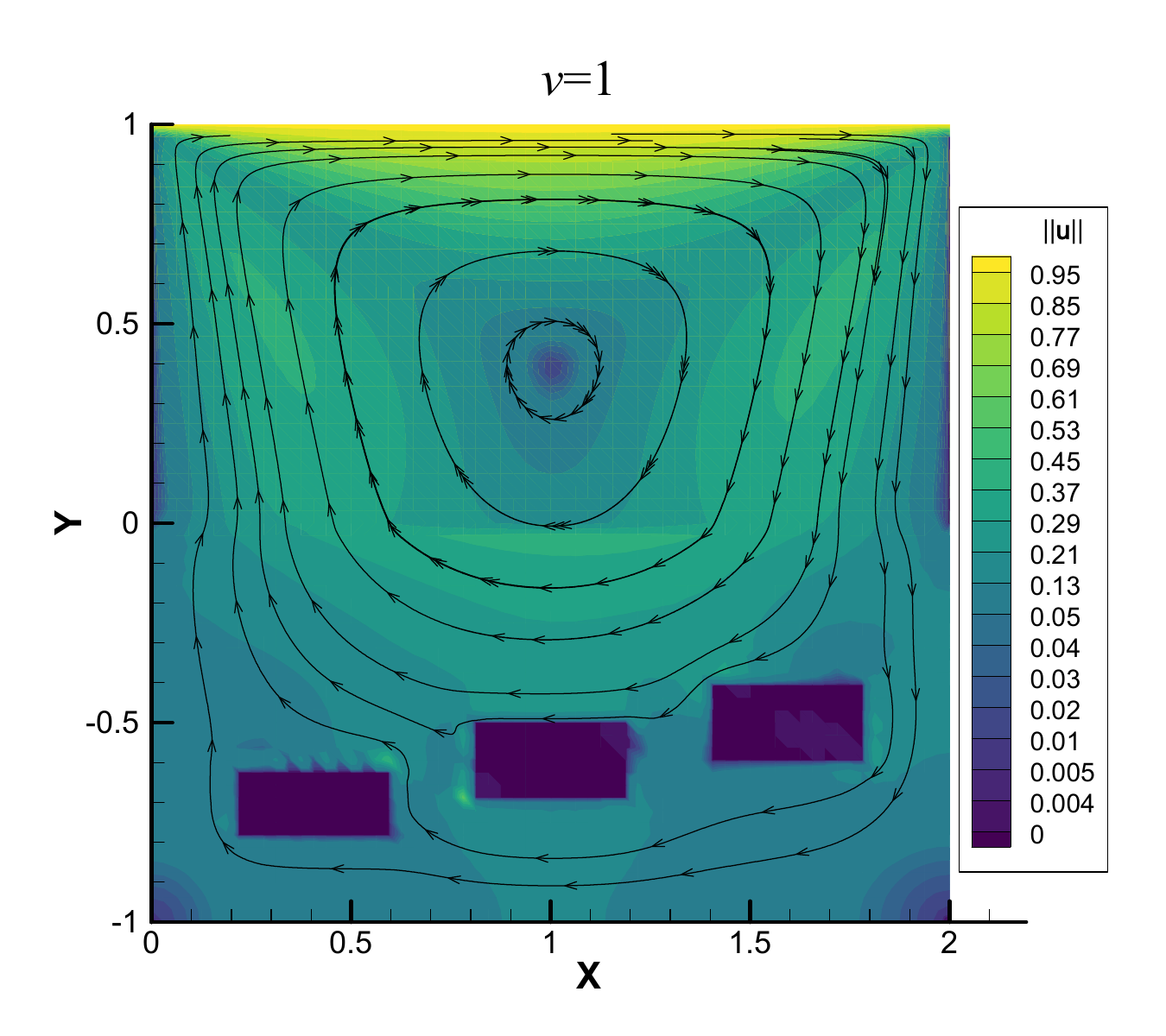}
        \label{ex5:nu=1SD}
    \end{minipage}
        \begin{minipage}{0.49\linewidth}
         \centering
        \includegraphics[scale=0.22]{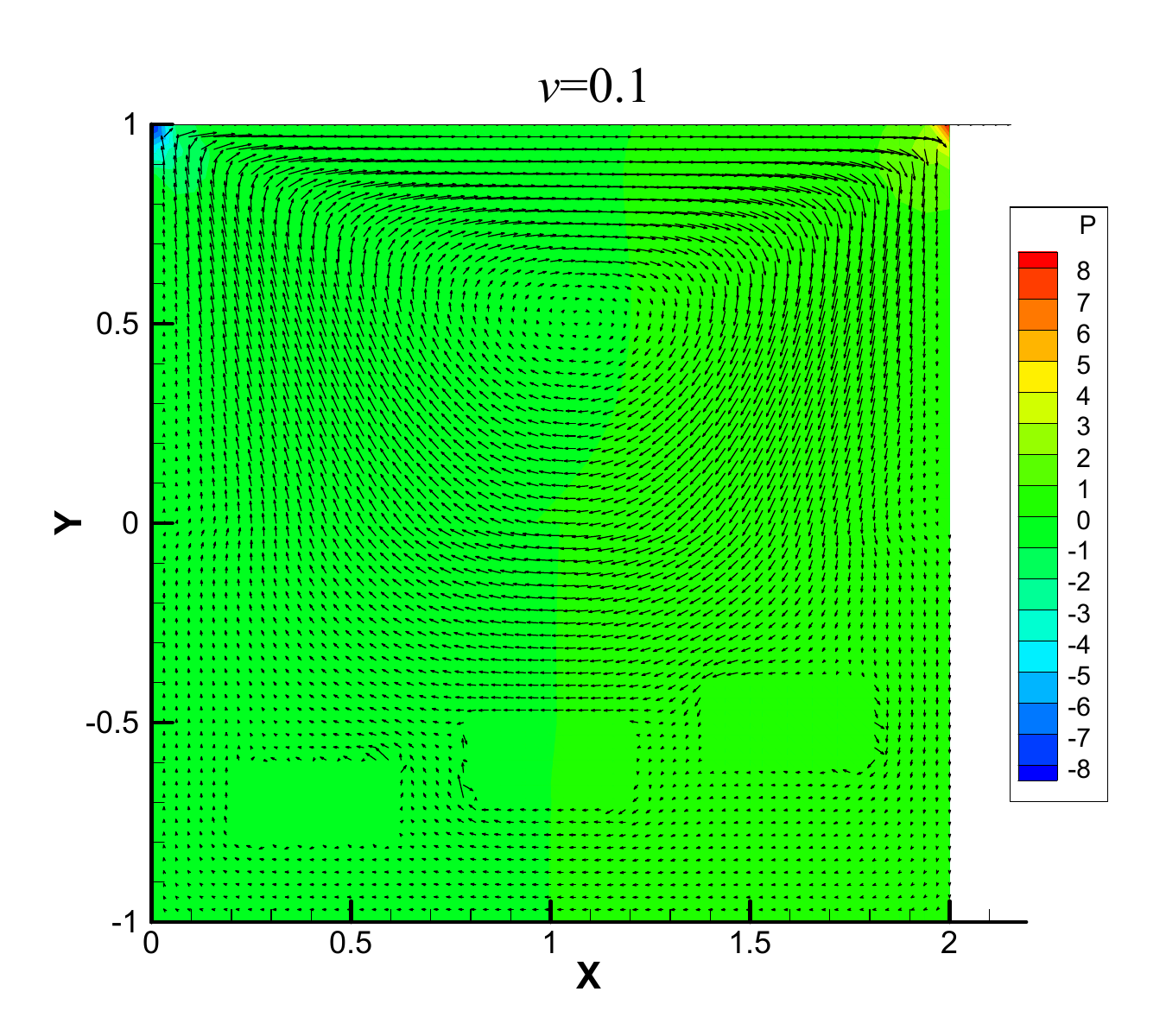}
        \label{ex5:nu=0.1DL}
    \end{minipage}
    \begin{minipage}{0.49\linewidth}
         \centering
        \includegraphics[scale=0.22]{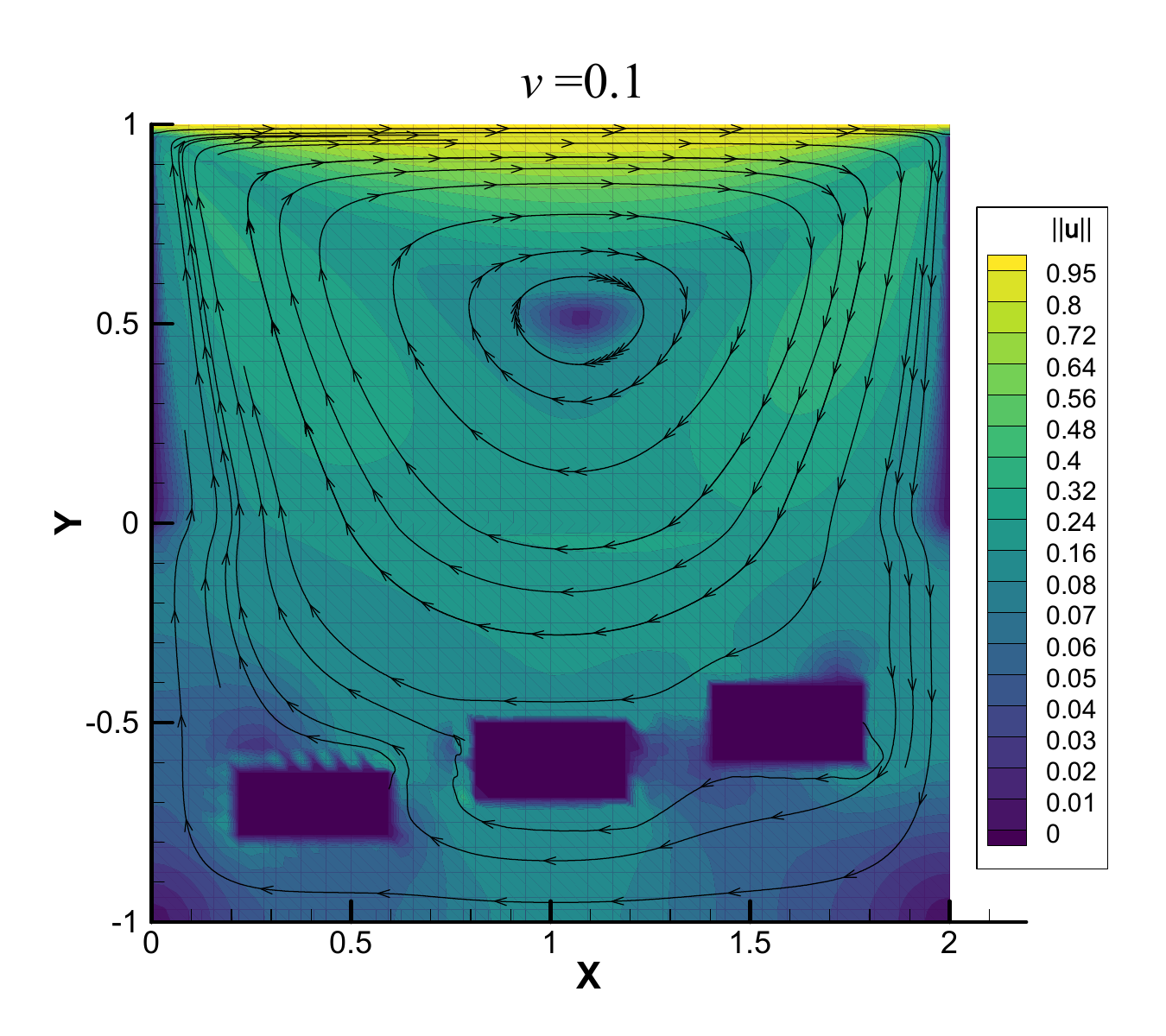}
    \end{minipage}
    \begin{minipage}{0.49\linewidth}
         \centering
        \includegraphics[scale=0.22]{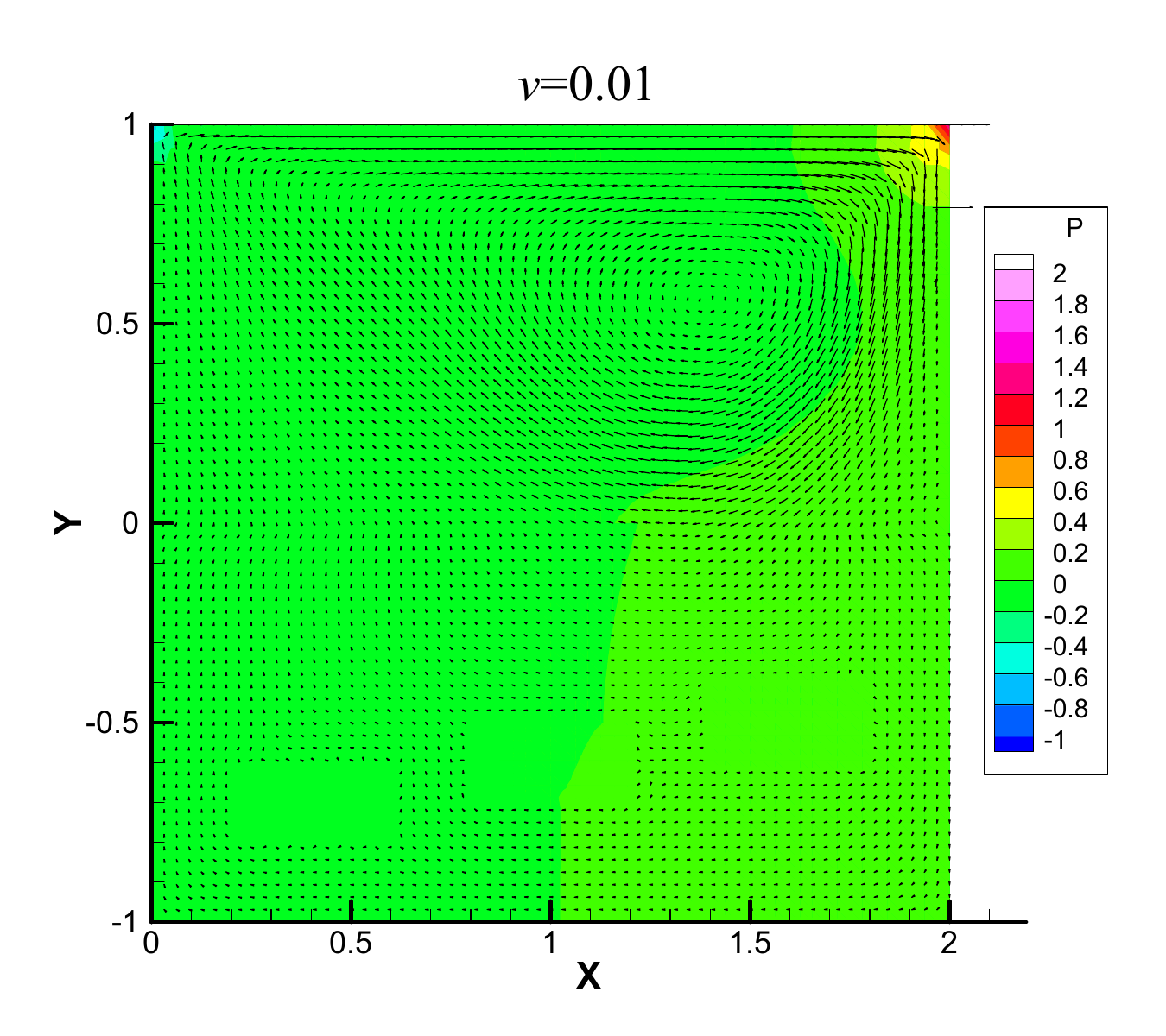}
        \label{ex5:nu=0.01DL}
    \end{minipage}
    \begin{minipage}{0.49\linewidth}
         \centering
        \includegraphics[scale=0.22]{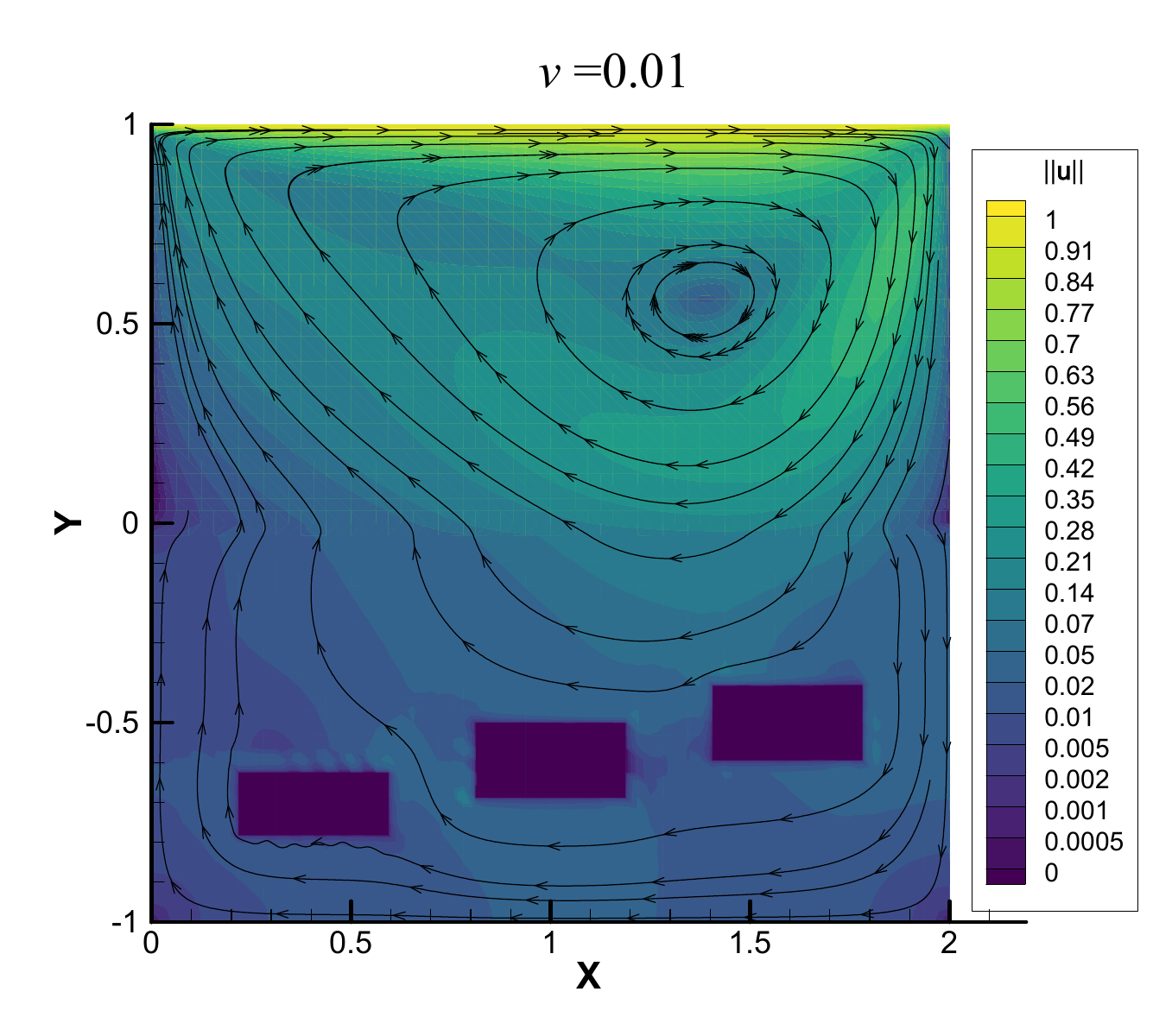}
        \label{ex5:nu=0.01SD}
    \end{minipage}
    \caption{ (Left) Numerical velocity and pressure under different initial conditions on a triangular mesh with h =$\frac{1}{16}$; (Right) The streamlines of velocity.}
    \label{fig:ex3}
\end{figure}

\subsection{A 3D-example with closed form solution}
 Finally, we test our algorithms for solving a three-dimensional problem with closed form solution. Let $\Omega \subset \mathbb{R}^{3}$ with $\Omega_{f}=(0,1) \times(0,1) \times(0,1), \Omega_{p}=(0,1)\times(0,1) \times$ $(1,2)$ and the interface $\Gamma=(0,1) \times\{1\}$. 
The exact solution $(\boldsymbol{u}, \varphi, p)$ is given by
\begin{equation*}
  \begin{cases}
        u= -(1-y)(1-z) & \text { in } \Omega_{f}, \\
        v= -(1-x)(1-z) & \text { in } \Omega_{f},\\
        w= (1-x)(1-y) & \text { in } \Omega_{f},\\
        p=(1-z)(1-x-y-z+4xyz) & \text { in } \Omega_{f},\\
        \varphi =(1-x)(1-y)(1-z) &\text { in } \Omega_{p},
  \end{cases}
\end{equation*}
where the components of $\boldsymbol{u}$ is denoted by $(u, v,w)$.
All the parameters except $\nu$ in the coupled model are set to
be 1. In \eqref{eq:3.2}, we set $\gamma_{1}=1, \gamma_{2} = 1, B_{\bs{u}}(\bx)=xyz(x-1)(y-1) $ and $B_{\varphi}(\boldsymbol{x})= xy(x-1)(y-1)(z-2)$. 
The uniform triangulation $\mathcal{T}_h$ with mesh size $h=\frac{1}{32}$ is adopted to be the test locations.


Similar to the 2D situation, we denote the components of $\bs{u}^{DL}$ obtained by Algorithm \ref{alg:DL} as $(u^{DL},$ $v^{DL},w^{DL})$ and test the computational performance of Algorithm \ref{alg:DL} and Algorithm \ref{alg:int-deep}.  From Table \ref{tab:3dDL-nu=1}$-$ Table \ref{tab:3dDL-nu=0.01}, we observe that the absolute discrete maximum errors of $u, v$, $w$, $p$ and $\varphi$ can reach $O(10^{-3})$ or $O(10^{-4})$. However, decreasing the iteration steps cannot further improve accuracy. In order to illustrate the accuracy of the Int-Deep method, we compare the iteration steps $\#\mathrm{K}$ of Newton iterative method employing different initial guesses with $h = \frac{1}{2^k}$, $k = 2, 3, 4$. The corresponding results are presented in Tables \ref{tab:iter3d-nu=1} $-$ \ref{tab:iter3d-nu=0.01}. As we can see, when the viscosity coefficient $\nu$ decreases, only the Int-Deep method can converge with few iteration steps.

%

\begin{table}[H]
\centering
\setlength{\tabcolsep}{2.5pt}
\scriptsize{
\begin{tabular}{cccccc}
\hline
$\#$Epoch & $||u-u^{DL}||_{0,\infty,h}$ & $||v-v^{DL}||_{0,\infty,h}$ & $||w-w^{DL}||_{0,\infty,h}$ & $||p-p^{DL}||_{0,\infty,h}$ & $||\varphi-\varphi^{DL}||_{0,\infty,h}$  \\
\hline
200 & 1.1280e-01 & 1.1474e-01 & 1.2573e-01 & 9.8521e-01 & 1.9437e-01 \\
500 & 5.8955e-03 & 6.0989e-03 & 4.6234e-03 & 1.8171e-01 & 7.3369e-03 \\
1000 & 3.9364e-03 & 4.6433e-03 & 3.1881e-03 & 6.4626e-02 & 4.4960e-03 \\
1500 & 8.2522e-03 & 8.5444e-03 & 8.1683e-03 & 2.9225e-02 & 2.2214e-03 \\
2000 & 1.0683e-03 & 1.0316e-03 & 7.5461e-04 & 1.5038e-02 & 9.4810e-04 \\
2500 & 2.4907e-03 & 2.4303e-03 & 2.6002e-03 & 1.2858e-02 & 8.1199e-04 \\
2500 & 2.4907e-03 & 2.4303e-03 & 2.6002e-03 & 1.2858e-02 & 8.1199e-04 \\
3000 & 7.7432e-04 & 1.1575e-03 & 1.1627e-03 & 8.1957e-03 & 1.2877e-03 \\
3500 & 4.7258e-04 & 6.2618e-04 & 4.2063e-04 & 6.4445e-03 & 8.3823e-04 \\
5000 & 3.6597e-04 & 6.0589e-04 & 3.5250e-04 & 7.9574e-03 & 6.8734e-04 \\
7500 & 2.8277e-04 & 4.7100e-04 & 3.0724e-04 & 3.4871e-03 & 4.7562e-04 \\
10000 & 1.3952e-04 & 3.0117e-04 & 2.9271e-04 & 3.3414e-03 & 3.3065e-04 \\
15000 & 1.6717e-04 & 2.4263e-04 & 2.3134e-04 & 2.9260e-03 & 2.9287e-04 \\
20000 & 1.3941e-04 & 2.2699e-04 & 2.1378e-04 & 2.5956e-03 & 2.9260e-04 \\
\hline
\end{tabular}}
    \caption{The absolute discrete maximum errors of Algorithm\ref{alg:DL}  with $\nu=1,\kappa=1$.}
    \label{tab:3dDL-nu=1}
\end{table}

\begin{table}
\centering
\setlength{\tabcolsep}{2.5pt}
\scriptsize{
\begin{tabular}{cccccc}
\hline
$\#$Epoch & $||u-u^{DL}||_{0,\infty,h}$ & $||v-v^{DL}||_{0,\infty,h}$ & $||w-w^{DL}||_{0,\infty,h}$ & $||p-p^{DL}||_{0,\infty,h}$ & $||\varphi-\varphi^{DL}||_{0,\infty,h}$  \\
\hline
200 & 4.9421e-02 & 3.9340e-02 & 1.1578e-01 & 3.6239e-02 & 4.6330e-02 \\
500 & 2.6120e-02 & 1.6990e-02 & 3.8038e-02 & 8.8871e-03 & 1.4511e-02 \\
1000 & 1.7341e-02 & 1.1687e-02 & 2.1754e-02 & 8.6221e-03 & 8.6668e-03 \\
1500 & 1.3706e-02 & 7.0937e-03 & 1.6307e-02 & 2.8401e-03 & 4.4740e-03 \\
2000 & 1.2608e-02 & 5.5066e-03 & 1.3255e-02 & 2.1570e-03 & 3.2142e-03 \\
2500 & 9.4278e-03 & 5.1800e-03 & 1.1713e-02 & 2.0543e-03 & 2.5913e-03 \\
2500 & 9.4278e-03 & 5.1800e-03 & 1.1713e-02 & 2.0543e-03 & 2.5913e-03 \\
3000 & 8.6336e-03 & 3.9181e-03 & 9.9709e-03 & 2.0516e-03 & 2.3925e-03 \\
3500 & 7.9073e-03 & 3.6801e-03 & 9.5410e-03 & 2.6332e-03 & 1.8229e-03 \\
5000 & 5.8340e-03 & 3.3927e-03 & 7.2748e-03 & 7.1876e-04 & 1.3675e-03 \\
7500 & 3.9786e-03 & 3.3174e-03 & 5.5008e-03 & 3.2889e-04 & 1.0128e-03 \\
10000 & 3.3260e-03 & 2.7889e-03 & 4.4905e-03 & 2.9798e-04 & 8.3913e-04 \\
15000 & 2.7782e-03 & 2.3925e-03 & 3.9075e-03 & 1.8952e-04 & 6.8426e-04 \\
20000 & 2.6754e-03 & 2.3359e-03 & 3.8142e-03 & 2.1757e-04 & 7.0270e-04 \\
\hline
\end{tabular}}
    \caption{The absolute discrete maximum errors of Algorithm \ref{alg:DL}  with $\nu=0.01,\;\kappa=1$.}
    \label{tab:3dDL-nu=0.01}
\end{table}
\begin{table}[H]
     \centering
     \begin{tabular}{c|cccc}\hline
       \diagbox[width=12em]{Initial value}{Mesh size}  & $h=\frac{1}{4}$&$h=\frac{1}{8}$ &$h=\frac{1}{16}$\\\hline
    $\boldsymbol{u}_{0}^{SD}$ &4 &4   &4  \\
    $\boldsymbol{0}$ &5 &5   &5  \\
    $\boldsymbol{1}$ &5 &5   &5  \\
    $\boldsymbol{u}_{0}^{DL}$ &4 &4   &4 \\\hline
     \end{tabular}
    \caption{The number of iterations $\#$K  under different initial values with $\nu=1,\;\kappa=1$.}
     \label{tab:iter3d-nu=1}
\end{table}
\begin{table}[H]
     \centering
     \begin{tabular}{c|cccc}\hline
       \diagbox[width=12em]{Initial value}{Mesh size}  & $h=\frac{1}{4}$&$h=\frac{1}{8}$ &$h=\frac{1}{16}$& \\\hline
    $\boldsymbol{u}^{SD}_{0}$ &7 &7   &7  \\
    $\boldsymbol{0}$ &8 &8   &8  \\
    $\boldsymbol{1}$ &8 &8   &8  \\
    $\boldsymbol{u}^{DL}_{0}$ &5 &5   &5  \\\hline
     \end{tabular}
    \caption{The number of iterations $\#$K  under different initial values with $\nu=0.01,\;\kappa=1$.}
     \label{tab:iter3d-nu=0.01}
     \end{table}

Table \ref{tab:3derr-nu=1K=1}$-$Table \ref{tab:3derr-nu=0.01K=1} show that the orders  of relative errors in $L^{2}$ norm for $\bs{u},p$ and $\varphi$ are $O(h^{3.7}),O(h^{2})$ and $O(h^{3}),$  respectively, and the orders of relative errors in $H^{1}$ norm for $\bs{u}$ and $\varphi$ are $O(h^{2.7})$ and $O(h^{2}),$ respectively, which is consistent with theoretical results.

\begin{table}[H]\scriptsize
\setlength{\tabcolsep}{2pt}
\centering
     \begin{tabular}{ccccccccccc}\hline
        &  $\|\boldsymbol{e}_{\bs{u}}^h\|_0$ &order & $\|e_{p}^h\|_0$ &order  & $\|e^h_{\varphi}\|_0$  &order  & $\|\boldsymbol{e}_{\bs{u}}^h\|_1$ &order  &  $\|e^h_{\varphi}\|_1$    &order  \\\hline
        h$=\frac{1}{4}$ &1.2611e-04 &-      & 1.4169e-03   &-      &7.8653e-02   &-  & 7.5859e-04  & -  & 1.3749e-02&-  \\
         h$=\frac{1}{8}$ &9.6019e-06 &3.7152  & 1.9505e-02 &2.0117 & 9.3657e-05  & 3.0179  & 2.2331e-04 &2.6656   & 3.4647e-03& 1.9886 \\
         h$=\frac{1}{16}$ & 7.2685e-07&3.7236  & 4.8542e-03 &2.0065   & 1.1678e-05&3.0036  &  3.4100e-05&2.7112   & 8.6976e-04&1.9940 \\\hline
    \end{tabular}
    \caption{The errors of Algorithm \ref{alg:int-deep} with $\nu=1,\;\kappa=1$. }
    \label{tab:3derr-nu=1K=1}
     \begin{tabular}{ccccccccccc}\hline
        &  $\|\boldsymbol{e}_{\bs{u}}^h\|_0$ &order & $\|e_{p}^h\|_0$ &order  & $\|e^h_{\varphi}\|_0$  &order  & $\|\boldsymbol{e}_{\bs{u}}^h\|_1$ &order  &  $\|e^h_{\varphi}\|_1$    &order                          \\\hline
        h$=\frac{1}{4}$ &1.1766e-02  &-  & 7.8758e-02 &-& 7.9584e-04 &- & 1.3575e-01 &-      &1.3819e-02&- \\
         h$=\frac{1}{8}$ &8.8893e-04 &3.7264  &1.9508e-02 &2.0134 & 9.4602e-05&3.0725 & 2.1471e-02  &2.6605      & 3.4671e-03 &1.9948  \\
         h$=\frac{1}{16}$ & 6.6064e-05&3.7501 & 4.8543e-03&2.0067 &  1.1691e-05&3.0165 & 3.2525e-03 &2.7227     & 8.6985e-04  &1.9949 \\\hline
    \end{tabular}
    \caption{The errors of Algorithm \ref{alg:int-deep} with $\nu=0.01,\;\kappa=1$. }
    \label{tab:3derr-nu=0.01K=1}
\end{table}

\section{Conclusion}

In this paper, we are concerned with Newton's method and its hybrid with machine learning for solving Navier-Stokes Darcy model. We develop the convergence analysis of Newton's method for the coupled nonlinear problem with mixed element discretization. Then, we develop a deep learning initialized iterative (Int-Deep) method for Navier-Stokes Darcy model. In this method, we employ the neural network solution generated by a few training steps as the initial guess for Newton iterative method with the FEM discretized problem. This hybrid method can converge to the true solution with the accuracy of FEM method rapidly. Numerical experiments also show that the method is robust with respect to the underlying physical parameters.


\end{document}